\documentclass[onefignum,onetabnum]{siamart190516}

\newif\ifARXIV
\ARXIVtrue

\usepackage{enumitem}
\usepackage{cite}
\usepackage{makecell}
\usepackage{booktabs}

\setcellgapes{2pt}

\newsiamremark{remarkx}{Remark}
\newsiamremark{examplex}{Example}
\newsiamremark{assumptionx}{Assumption}
\newenvironment{example}{\examplex}{\oprocend\endexamplex}
\newenvironment{remark}{\remarkx}{\oprocend\endremarkx}
\newcommand\oprocendsymbol{\hbox{\small $\blacksquare$}}
\newcommand\oprocend{\relax\ifmmode\else\unskip\hfill\fi\oprocendsymbol}

\numberwithin{theorem}{section}
\numberwithin{equation}{section}

\crefname{examplex}{Example}{Examples}

\usepackage{amsfonts}
\usepackage{amssymb}

\DeclareMathOperator{\co}{co}
\DeclareMathOperator{\cocl}{\overline{co}}
\DeclareMathOperator{\cl}{cl}
\DeclareMathOperator{\gph}{gph}

\DeclareMathOperator{\grad}{grad}

\DeclareMathOperator{\epi}{epi}
\DeclareMathOperator{\minimize}{minimize}
\DeclareMathOperator{\subjto}{subject~to}

\newcommand{\tproj}[3]{\Pi_{#1}^{#2}[{#3}]}
\newcommand{\Kras}[1]{\operatorname{K}\left[{#1}\right]}
\newcommand{\minEig}[1]{\lambda^{\min}_{{#1}}}
\newcommand{\maxEig}[1]{\lambda^{\max}_{{#1}}}
\newcommand{\condN}[1]{\kappa_{{#1}}}

\newcommand{\bbR}{\mathbb{R}}

\newcommand{\calX}{\mathcal{X}}

\newcommand{\calM}{\mathcal{M}}
\newcommand{\calS}{\mathcal{S}}
\newcommand{\calN}{\mathcal{N}}

\newcommand{\calV}{\mathcal{V}}
\newcommand{\calW}{\mathcal{W}}
\newcommand{\calA}{\mathcal{A}}
\newcommand{\calU}{\mathcal{U}}

\usepackage{xcolor}
\usepackage{subcaption}
\usepackage{graphicx}
\usepackage{tikz}
\usepackage{epstopdf}
\usepackage{pgfplots}

\graphicspath{{./Figs/}}
\ifpdf
	\DeclareGraphicsExtensions{.eps,.pdf,.png,.jpg}
\else
	\DeclareGraphicsExtensions{.eps}
\fi

\usepgfplotslibrary{colorbrewer}
\usepgfplotslibrary{fillbetween}
\usetikzlibrary{arrows.meta}
\usetikzlibrary{cd}

\pgfplotsset{%
	scale only axis,
	width=1.28125in,
	height=1.5in,
	compat=newest,
}

\pgfmathsetmacro\s{50}

\pgfmathsetmacro\r{.5}
\pgfmathsetmacro\d{0.08}
\pgfmathsetmacro\a{0.60}

\pgfmathsetmacro\rr{{sqrt(( + \d^\a * \a * \d^(\a -1))^2 + (\d^\a)^2)}}
\newcommand{\watershedprox}[2]{
  \begin{tikzpicture}[thick,scale=(#2), every node/.style={scale=(#2)}]
        \begin{axis}
        [
            set layers=axis on top,
            view={0}{90},
            domain=-0.7:.9,
            domain y=-.75:.75,
            xmin=-.7, xmax=1,
            ticks=none,
            axis lines=middle,
            axis equal,
            enlargelimits=false,
            axis line style={on layer=axis background},
        ]         

        \addplot[fill, color=gray!15] (-.7, -.90) rectangle (0.0, .90);
        
        \addplot[name path=ceil, draw=none, samples=\s, domain=.05:.85]{.90};
        \addplot[name path=floor, draw=none, samples=\s, domain=.05:.85]{-.90};
        \addplot[name path=upper, thin, samples=\s, domain=.05:.85]{x^(#1)};    
        \addplot[name path=lower, thin, samples=\s, domain=.05:.85]{-x^(#1)};
        \addplot[fill, color=gray!15] fill between[of=ceil and upper];
        \addplot[fill, color=gray!15] fill between[of=lower and floor];   
        
        \addplot[name path=ceil, draw=none, samples=\s, domain=0:.05]{.90};
        \addplot[name path=floor, draw=none, samples=\s, domain=0:.05]{-.90};
        \addplot[name path=upper, thin, samples=\s, domain=0:.05]{x^(#1)};    
        \addplot[name path=lower, thin, samples=\s, domain=0:.05]{-x^(#1)};    
        \addplot[fill, color=gray!15] fill between[of=ceil and upper];
        \addplot[fill, color=gray!15] fill between[of=floor and lower];   
      
        \ifdim#1pt>0.5pt        	       	
	        
	        \draw[dotted] (axis cs: {\d + \d^#1 * #1 * \d^(#1 -1)}, 0) circle [
	            radius={sqrt(( + \d^#1 * #1 * \d^(#1 -1))^2 + (\d^#1)^2)}
	        ];
	        
	        \draw[<-] ({\d}, {\d^#1 }) -- (
	         {\d + \d^#1 * sqrt(1 + #1 ^2 * \d^(2*(#1 - 1 ))) *#1 * \d^(#1 -1) / sqrt(1 + #1 ^2 * \d^(2*(#1 - 1 )))}, 
	         { 0 }
	        );        
	        
	        \draw [<-] ({\d}, {-\d^#1  }) -- (
	         {\d + \d^#1 * sqrt(1 + #1 ^2 * \d^(2*(#1 - 1 ))) *#1 * \d^(#1 -1)   / sqrt(1 + #1 ^2 * \d^(2*(#1 - 1 )))}, 
	         { 0 }
	        );
		\else
		\draw[dashed] (axis cs: \rr, 0) circle [
	            radius=\rr
	        ];
        \fi
		
        \end{axis}    
  \end{tikzpicture}
}

\ifpdf
\hypersetup{
	pdftitle={Oblique Projected Dynamical Systems},
	pdfauthor={A. Hauswirth, S. Bolognani, and F. D\"orfler}
}
\fi

\headers{Oblique Projected Dynamical Systems}{A. Hauswirth, S. Bolognani, and F. D\"orfler}

\title{Projected Dynamical Systems on Irregular, Non-Euclidean Domains for Nonlinear Optimization\thanks{Submitted to the editors DATE.
\funding{This work was supported by ETH Zurich and the SNF AP Energy Grant \#160573.}}}

\author{
Adrian Hauswirth\thanks{Department of Information Technology and Electrical Engineering, ETH Z\"urich, Zurich, Switzerland (\email{hadrian@ethz.ch}, \email{bsaverio@ethz.ch}, \email{dorfler@ethz.ch}).
}
\and Saverio Bolognani\footnotemark[2]
\and Florian D\"orfler\footnotemark[2]
}

\begin{document}

\maketitle

\begin{abstract}
	Continuous-time projected dynamical systems are an elementary class of discontinuous dynamical systems with trajectories that remain in a feasible domain by means of projecting outward-pointing vector fields. They are essential when modeling physical saturation in control systems, constraints of motion, as well as studying projection-based numerical optimization algorithms. Motivated by the emerging application of feedback-based continuous-time optimization schemes that rely on the physical system to enforce nonlinear hard constraints, we study the fundamental properties of these dynamics on general locally-Euclidean sets. Among others, we propose the use of Krasovskii solutions, show their existence on nonconvex, irregular subsets of low-regularity Riemannian manifolds, and investigate how they relate to conventional Carath\'eodory solutions. Furthermore, we establish conditions for uniqueness, thereby introducing a generalized definition of prox-regularity which is suitable for non-flat domains. Finally, we use these results to study the stability and convergence of projected gradient flows as an illustrative application of our framework. We provide simple counter-examples for our main results to illustrate the necessity of our already weak assumptions.
\end{abstract}

\section{Introduction}

Projected dynamical systems form an important class of discontinuous dynamical systems whose trajectories remain in a domain $\calX$. This invariance (or \emph{viability}) of $\calX$ is achieved by projecting a vector field $f$ on the tangent cone of $\calX$.  More specifically, in the interior of $\calX$, trajectories follow the vector field $f$. At the boundary, instead of leaving $\calX$, trajectories ``slide'' along the boundary of $\calX$ in the feasible direction that is closest to the direction imposed by $f$. This qualitative behavior is illustrated in \cref{fig:qual_pds1}.

Even though projected dynamical systems have a long history in different contexts such as the study of variational inequalities or differential inclusions, new compelling applications in the context of real-time optimization require a different, more general approach. Hence, this paper is primarily motivated by the renewed interest in dynamical systems that solve optimization problems. Early works in this spirit such as~\cite{brockettDynamicalSystemsThat1988,helmkeOptimizationDynamicalSystems1996} have designed continuous-time systems to solve computational problems such as diagonalizing matrices or solving linear programs. This has further resulted in the study of optimization algorithms over manifolds~\cite{absilOptimizationAlgorithmsMatrix2008}. Recently, interest has shifted towards analyzing existing iterative schemes with tools from dynamical systems including Lyapunov theory~\cite{wilsonLyapunovAnalysisMomentum2016} and integral quadratic constraints~\cite{lessardAnalysisDesignOptimization2016,fazlyabAnalysisOptimizationAlgorithms2017}. Most of these have considered unconstrained optimization problems~\cite{suDifferentialEquationModeling2014} and algorithms that can be modelled with a standard ODE~\cite{kricheneAcceleratedMirrorDescent2015} or with variational tools~\cite{wibisonoVariationalPerspectiveAccelerated2016}. With this paper we hope to pave the way for the analysis of algorithms for constrained optimization whose continuous-time limits are discontinuous.

Recently, this idea of studying the dynamical aspects of optimization algorithms has given rise to a new type of feedback control design that aims at steering a physical system in real time to the solution of an optimization problem~\cite{nelsonIntegralQuadraticConstraint2018,zhangDistributedControlReaching2018,mentaStabilityDynamicFeedback2018,colombinoOnlineOptimizationFeedback2019,bernsteinOnlinePrimalDualMethods2019} without external inputs.
Precursors of this idea have been used in the analysis of congestion control in communication networks~\cite{kellyRateControlCommunication1998,lowInternetCongestionControl2002}. More recently, the concept has been widely applied to power systems~\cite{hauswirthOnlineOptimizationClosed2017,dallaneseOptimalPowerFlow2018,molzahnSurveyDistributedOptimization2017,tangRealTimeOptimalPower2017}. This context is particularly challenging, because the physical laws of power flow, saturating components, and other constraints define a highly non-linear, nonconvex feasible domain.

Projected dynamical systems provide a particularly useful framework to model actuation constraints and physical saturation in this context, but existing results are of limited applicability for complicated problems. Hence, in this paper, we consider new, generalized features for projected dynamical systems. We consider for example \emph{irregular} feasible domains (\cref{fig:qual_pds2}) for which traditional \emph{Carath\'eodory solutions} can fail to exist or may not be unique. Furthermore, \emph{non-orthogonal projections} occur in non-Euclidean spaces and may alter the dynamics.
Finally, coordinate-free definitions are required to study projected dynamical systems on subsets of manifolds (\cref{fig:qual_pds4}).

\begin{figure}[tb]
	\centering
	\begin{subfigure}[t]{0.25\textwidth}
		\includegraphics[width=\textwidth]{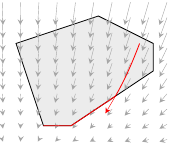}
		\caption{}\label{fig:qual_pds1}
	\end{subfigure} \hspace{.015\textwidth}
	\begin{subfigure}[t]{0.25\textwidth}
		\includegraphics[width=\textwidth]{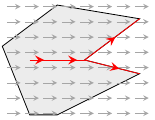}
		\caption{}\label{fig:qual_pds2}
	\end{subfigure} \hspace{.015\textwidth}
	\begin{subfigure}[t]{0.25\textwidth}
		\centering
		\includegraphics[width=.75\textwidth]{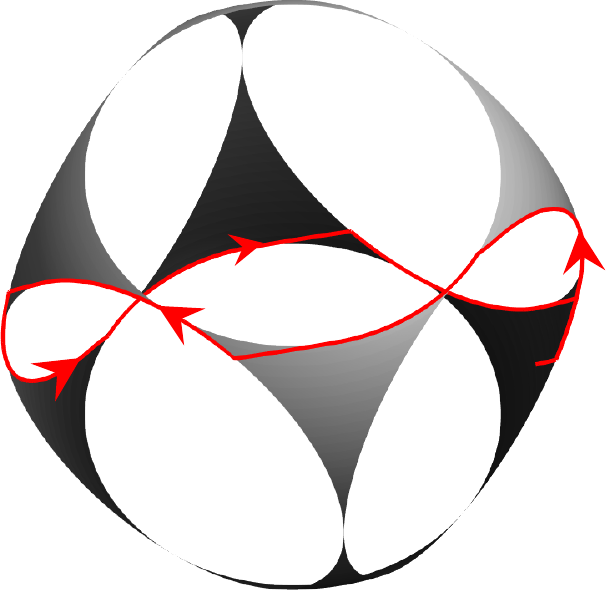}
		\caption{}\label{fig:qual_pds4}
	\end{subfigure}
	\caption{Qualitative behavior of projected dynamical systems: (a) projected gradient flow on a convex polyhedron, (b) flow on an irregular set with non-unique trajectory, (c) periodic projected trajectory on a subset of a sphere.}\label{fig:qual_pds}
\end{figure}

\subsection*{Literature review}
Different approaches have been reviewed and explored to establish the results in this paper.
One of the earliest formulations of projected dynamical systems goes back to~\cite{henryExistenceTheoremClass1973} which establishes the existence of Carath\'eodory solutions on closed convex domains.
In~\cite{cornetExistenceSlowSolutions1983} this requirement is relaxed to $\calX$ being Clarke regular (for existence) and prox-regular (for uniqueness).
In the larger context of differential inclusions and viability theory~\cite{aubinViabilityTheory1991,aubinDifferentialInclusionsSetValued1984}, projected dynamical systems are often presented as specific examples of more general differential inclusions, but without substantially generalizing the results of~\cite{henryExistenceTheoremClass1973,cornetExistenceSlowSolutions1983}.
In the context of variational equalities,~\cite{nagurneyProjectedDynamicalSystems1996} provides alternative proofs of existence and uniqueness of Carath\'eodory solutions when the domain $\calX$ is a convex polyhedron by using techniques from stochastic analysis.
In~\cite{brogliatoEquivalenceComplementaritySystems2006}, various equivalence results between the different formulations are established for convex $\calX$.
Finally, projected dynamical systems have been defined and studied in the more general context of Hilbert~\cite{cojocaruExistenceSolutionsProjected2004} and Banach spaces~\cite{cojocaruNonpivotImplicitProjected2012,giuffreClassesProjectedDynamical2008}. The latter, in particular, is complicated by the lack of an inner product and consequently more involved projection operators~\cite{alberGeneralizedProjectionOperators1996}.

The behavior of projected dynamical systems as illustrated in \cref{fig:qual_pds} suggests the presence of switching mechanics that result in different vector fields being active in different parts of the domain and its boundary in particular. This idea is further supported by the fact that in the study of optimization problems with a feasible domain delimited by explicit constraints, it is often useful to define the (finite) \emph{set of active constraints} at a given point. This suggests that projected dynamical systems should be modeled as switched~\cite{liberzonSwitchingSystemsControl2003} or even hybrid systems~\cite{goebelHybridDynamicalSystems2012} or hybrid automata~\cite{lygerosDynamicalPropertiesHybrid2003,simicGeometricTheoryHybrid2000}. However, projected dynamical systems are much more easily (and generally) modeled as differential inclusions without explicitly considering any type of switching.

A special case of projected dynamical systems are subgradient and saddle-point flows arising in non-smooth and constrained optimization. Whereas projection-based algorithms and subgradients are ubiquitous in the analysis of iterative algorithms, work on their continuous-time counterparts is far less prominent has only been studied with limited generality~\cite{arrowStudiesLinearNonlinear1958,cherukuriAsymptoticConvergenceConstrained2016,cortesDiscontinuousDynamicalSystems2008,hauswirthProjectedGradientDescent2016}, e.g., restricted to convex problems.

\subsection*{Contributions}
In this paper, we study a generalized class of projected dynamical systems in finite dimensions that allows for oblique projection directions. These variable projection directions are described by means of a (possibly non-differentiable) metric $g$ and are essential in providing a coordinate-free definition of projected dynamical systems on low-regularity Riemannian manifolds. Compared to previous work, we do not make a-priori assumptions on the regularity (or convexity) of the feasible domain $\calX$ or the vector field $f$. Instead, we strive to illustrate the necessity of those assumptions that we require by a series of (non-)examples.

Our main contribution is the development of a self-contained and comprehensive theory for this general setup. Namely, we provide weak requirements on the feasible set $\calX$, the vector field $f$, the metric $g$ and the differentiable structure of the underlying manifold that guarantee existence and uniqueness of trajectories, as well as other properties. \cref{tab:summary} at the end of the paper concisely summarizes these results.

To be able to work with projected dynamical systems on irregular domains and with discontinuous vector fields, we resort to so-called \emph{Krasovskii solutions} that are a weaker notion than the classical \emph{Carath\'eodory solutions} and are commonly used in the study of differential inclusions because their existence is guaranteed under minimal requirements. We derive this set of regularity conditions in the specific context of projected dynamical system.
Under slightly stronger assumptions involving continuity and Clarke regularity, we show that Krasovskii solutions coincide with the classical Carath\'eodory solutions, thus recovering (in case of the Euclidean metric) known requirements for the existence of the latter. Finally, we lay out the requirements for uniqueness of solutions which are based on Lipschitz-continuity and a new, generalized definition of prox-regularity which is suitable for low-regularity Riemannian manifolds, i.e., manifolds that do not necessarily have a $C^\infty$ structure~\cite{hosseiniMetricProjectionProxregular2013,bernicotSweepingProcessProxregular2015}. Our already weak regularity conditions are sharp in the sense that counter-examples can be constructed to show that requirements cannot be violated individually without the respective result failing to hold.

A major appeal of our analysis framework is its geometric nature: All of our notions are preserved by sufficiently regular coordinate transformations, which allows us to extend all of our results to constrained subsets of differential manifolds. A noteworthy by-product of this analysis is the fact that our generalized definition of prox-regularity is an intrinsic property of subsets of $C^{1,1}$ manifolds, i.e., independent of the metric, even though the traditional definition (on $\bbR^n$) suggests that prox-regularity depends on the choice of metric.

Through a series of examples, we demonstrate the application of our framework to general (nonlinear and nonconvex) optimization problems and study the stability and convergence of projected gradient dynamics under very weak regularity assumptions.

Thus, we believe that our results are not only of interest in the context of discontinuous dynamical systems, but we also envision their use in the analysis of algorithms for nonlinear, nonconvex optimization problems, possibly on manifolds.
The properties developed in the present paper also form a solid foundation for constrained feedback control and online optimization in various contexts. Some preliminary results for online optimization in power systems can be found in~\cite{hauswirthProjectedGradientDescent2016,hauswirthOnlineOptimizationClosed2017}.

\subsection*{Paper organization}
After introducing notation and preliminary definitions in \cref{sec:preliminaries,sec:pds}, we establish the existence of Krasovskii solutions to projected dynamical systems on $\bbR^n$ in \cref{sec:exist}. \Cref{sec:equiv} establishes equivalence of Krasovskii and Carath\'eodory solutions under Clarke regularity and we point out the connection to related work. In \cref{sec:uniq}, we elaborate on the requirements for uniqueness and in \cref{sec:mfd} we define projected dynamical systems on low-regularity Riemannian manifolds and establish the requirements on the differentiable structure that guarantee existence and uniqueness.  As an illustration of optimization applications, in \cref{sec:stab} we consider Krasovskii solutions of projected gradient systems on irregular domains, we study their convergence and stability and revisit the connection to subgradient flows.
Throughout the paper, we illustrate our theoretical developments with insightful examples. Finally, \cref{sec:conclusion} concisely summarizes our results in the form of \cref{tab:summary} and concludes the paper.
The appendix includes technical definitions and results that are used in proofs but are not required to understand the main results of the paper.
\unless\ifARXIV
	Some lengthy algebraic manipulations and technical, though standard, proofs are only available online, in the extended version of this paper~\cite{hauswirthProjectedDynamicalSystems2018a}.
\fi

\section{Preliminaries}\label{sec:preliminaries}
\subsection{Notation}
We only consider finite-dimensional spaces. Unless explicitly noted otherwise, we will work in the usual Euclidean setup for $\bbR^n$ with inner product $\left\langle \cdot, \cdot \right\rangle$ and 2-norm $\| \cdot \|$. Whenever it is informative, we make a formal distinction between $\bbR^n$ and its tangent space $T_x \bbR^n$ at $x \in \bbR$, even though they are isomorphic.
For a set $A \subset \bbR^n$ we use the notation $\| A \| := \sup_{v \in A} \| v \|$. The closure, convex hull and closed convex hull of $A$ are denoted by $\cl A$, $\co A$, and $\cocl A$, respectively.
The set $A$ is \emph{locally compact} if it is the intersection of a closed and an open set. A neighborhood $U\subset A$ of $x \in A$ is understood to be a \emph{relative neighborhood}, i.e., with respect to the subspace topology on $A$.
Given a convergent sequence $\{x_k\}$, the notation $x_k \underset{A} \rightarrow x$ implies that $x_k \in A$ for all $k$. If $x_k \in \bbR$, the notation $x \rightarrow 0^+$ means $x_k > 0$ for all~$k$ and $x_k$ converges to 0.

Let $V$ and $W$ be vector spaces endowed with norms $\| \cdot \|_V$ and $\| \cdot \|_W$, respectively, and let $A \subset V$. Continuous maps $\Phi: A \rightarrow W$ are denoted by $C^0$. The map $\Phi$ is \emph{(locally) Lipschitz} (denoted by $C^{0,1}$) if for every $x \in A$ there exists $L>0$ such that for all $z, y \in A$ in a neighborhood of $x$ it holds that
\begin{equation}\label{eq:def_lipschitz}
	\| \Phi(z) - \Phi(y) \|_W \leq L \| z - y \|_V \, .
\end{equation}
The map $\Phi$ is \emph{globally Lipschitz} if~\eqref{eq:def_lipschitz} holds for the same $L$ for all $z,y$.

Differentiability is understood in the sense of Fr\'echet. Namely, if $A$ is open, then the map $\Phi$ is \emph{differentiable at $x$} if there is a linear map $D_x \Phi: V \rightarrow W$ such that
\begin{equation*}
	\underset{y \rightarrow x}{\lim} \, \frac{\| \Phi(y) - \Phi(x) - D_x \Phi (y - x) \|_W}{\| y - x \|_V} = 0 \, .
\end{equation*}
The map $\Phi$ is \emph{differentiable} ($C^1$) if it is differentiable at every $x \in A$.  It is $C^{1,1}$ if it is $C^1$ and $D_x \Phi$ is $C^{0,1}$ (as function of $x$). Finally, given bases for $V$ ($\dim V= m$) and $W$ ($\dim W = n$), the \emph{Jacobian of $\Phi$ at $x$} is denoted by the $n\times m$-matrix $\nabla \Phi(x)$.

In our context, a \emph{set-valued map} $F: A \rightrightarrows \bbR^n$ where $A \subset \bbR^n$ is a map that assigns to every point $x \in A$ a set $F(x) \subset T_x\bbR^n$.
The set-valued map $F$ is \emph{non-empty}, \emph{closed}, \emph{convex}, or \emph{compact} if for every $x\in A$ the set $F(x)$ is non-empty, closed, convex, or compact, respectively.
It is \emph{locally bounded} if for every $x \in A$ there exists $L > 0$ such that $\| F(y) \| \leq L$ for all $y \in A$ in a neighborhood of $x$. The same definition also applies to single-valued functions. The map $F$ is \emph{bounded} if there exists $L > 0$ such that $\| F(y) \| \leq L$ for all $x \in A$. The \emph{inner} and \emph{outer limits} of $F$ at $x$ are denoted by $\lim \inf_{y \rightarrow x} F(y)$ and $\lim \sup_{y \rightarrow x} F(y)$ respectively (see appendix for a formal definition and summary of continuity concepts which are required for certain proofs only).

\subsection{Tangent and Clarke Cones}
The ensuing definitions follow~\cite[Chap.~6]{rockafellarVariationalAnalysis2009}.

\begin{definition}\label{def:tgt_cone}
	Given a set $\calX \subset \bbR^n$ and $x \in \calX$, a vector $v \in T_x\bbR^n$ is a \emph{tangent vector of $\calX$ at $x$} if there exist sequences $x_k \underset{\calX}{\rightarrow} x$ and $\delta_k \rightarrow 0^+$ such that $\tfrac{x_k - x}{\delta_k} \rightarrow v$. The set of all tangent vectors is the \emph{tangent cone of $\calX$ at $x$} and denoted by $T_x \calX$.
\end{definition}

The tangent cone $T_x \calX$ (also known as \emph{(Bouligand's) contingent cone~\cite{clarkeNonsmoothAnalysisControl1998}}) is closed and non-empty (namely, $0 \in T_x \calX$) for any $x \in \calX$.

In the following definition of Clarke regularity and in most of paper we limit ourselves to locally compact subsets of $\bbR^n$. In our context, a more general definition of Clarke regularity does not improve our results and only adds to the technicalities.

\begin{definition}\label{def:clarke_tgt}
	For a locally compact set $\calX \subset \bbR^n$ the \emph{Clarke tangent cone at $x \in \calX$} is defined as the inner limit of the tangent cones, i.e., $T^C_x \calX := \underset{y \rightarrow x}{\lim \inf} \, T_y \calX$.
\end{definition}

By definition of the inner limit, we have $T^C_x \calX \subseteq T_x \calX$.
Furthermore, $T^C_x \calX$ is closed, convex and non-empty for all $x \in \calX$~\cite[Thm.~6.26]{rockafellarVariationalAnalysis2009}.

\begin{definition}\label{def:clarke_reg} We call a set $\calX \subset \bbR^n$ \emph{Clarke regular at $x$} if it is locally compact and $T_x \calX = T_x^C \calX$. The set $\calX$ is \emph{Clarke regular} if it is Clarke regular for all $x \in \calX$.
\end{definition}

\cref{fig:tgt_cone} illustrates the definition of a tangent vector by a sequence $\{x_k \}$ that approaches $x$ in a tangent direction. \cref{fig:ctgt_cone} shows a set that is not Clarke regular.

The following example illustrates that, under standard constraint qualifications as used in optimization theory, sets defined by $C^1$ inequality constraints are Clarke regular. Such sets are generally encountered in nonlinear programming.

\begin{figure}[tb]
	\centering
	\begin{subfigure}[t]{0.25\textwidth}
		\centering
		\includegraphics[height=2.5cm]{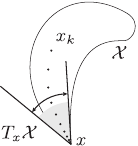}
		\subcaption{}\label{fig:tgt_cone}
	\end{subfigure} \hspace{.08cm}
	\begin{subfigure}[t]{0.25\textwidth}
		\centering
		\includegraphics[height=2.5cm]{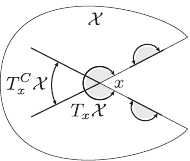}
		\subcaption{}\label{fig:ctgt_cone}
	\end{subfigure} \hspace{.1cm}
	\begin{subfigure}[t]{0.25\textwidth}
		\centering
		\includegraphics[height=2.3cm]{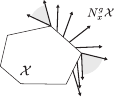}
		\subcaption{}\label{fig:norm_cone}
	\end{subfigure}
	\caption{Tangent cone construction (a), Clarke tangent cone at an irregular point (b), and oblique normal cones induced by a non-Euclidean metric (c).}
\end{figure}

\begin{example}[sets defined by inequality constraints]\label{ex:clarke_reg_constraint_set}
	Let $h: \bbR^n \rightarrow \bbR^m$ be $C^{1}$ such that $\nabla h(x)$ has full rank for all $x$.\footnote{This rank condition is a standard \emph{constraint qualification} in nonlinear programming~\cite{bazaraaNonlinearProgrammingTheory2006}. In general, instead of $\nabla h(x)$ having full rank for all $x$, it suffices that for a given $x$ only the active constraints (i.e., $\nabla h_{I(x)}(x)$) have full rank. Furthermore, equality constraints can be easily incorporated.}
	Then, the set $\calX := \{ x \, | \, h(x) \leq 0 \}$ is Clarke regular~\cite[Thm.~6.31]{rockafellarVariationalAnalysis2009}. In particular, let $h$ be expressed componentwise as $h(x) = {\left[ h_1(x), \ldots, h_m(x) \right]}^T$, let $I(x) := \{ i \, | \, h_i(x) = 0 \}$ denote the set of active constraints at $x \in \calX$ and define $h_{I(x)} := {[ h_i(x) ]}_{i\in I(x)}$ as the function obtained from stacking the active constraint functions. Then, the (Clarke) tangent cone at $x$ in the canonical basis is given by
	$T^C_ x \calX = T_x \calX = \{ v \, | \, \nabla h_{I(x)}(x) v \leq 0 \}$.
\end{example}

\subsection{Low-regularity Riemannian metrics}
A natural extension for projected dynamical systems are oblique projection directions. These are conveniently defined via a (Riemannian) metric which defines a variable inner product on $T_x \bbR^n$ as function of $x$. Furthermore, the notion of a Riemannian metric is essential to define projected dynamical systems in a coordinate-free setup on manifolds.

We quickly review the definition of bilinear forms and inner products. Let $L_2^n$ denote the space of bilinear forms on $\bbR^n$, i.e., every $g \in L_2^n$ is a map $g: \bbR^n \times \bbR^n \rightarrow \bbR$ such that for every $u, v, w \in \bbR^n$ and $\lambda \in \bbR$ it holds that $g(u + v, w) = g(u, w) + g(v, w)$ and $g(u, v + w) = g(u, v) + g(u, w)$ as well as $g(\lambda v, w) = \lambda g(v, w) = g(v, \lambda w)$. Given the canonical basis of $\bbR^n$, $g$ can be written in matrix form as $g(u, v) := u^T G v$ where $G \in \bbR^{n \times n}$. In particular, $L_2^n$ is itself a $n^2$-dimensional space isomorphic to $\bbR^{n \times n}$.

An \emph{inner product} $g \in L_2^n$ is a symmetric, positive-definite bilinear form, that is, for all $u, v \in \bbR^n$ we have $g(u, v) = g(v,u)$. Further, $g(u,u) \geq 0$, and $g(u,u) = 0$ holds if and only if $u = 0$. If $g$ is an inner product we use the notation $\left\langle u, v \right\rangle_g := g(u,v)$. In matrix form, we can write $\left\langle u, v \right\rangle_g := u^T G v$ where $G$ is symmetric positive definite.

We write $\|\cdot \|_{g}$ given by $\| v \|_{g} := \sqrt{\left\langle v, v \right\rangle_{g}}$ to denote the 2-norm induced by $g$. The \emph{maximum} and \emph{minimum eigenvalues} of $g$ are denoted by $\maxEig{g} : = \max \{ \| v \|_{g} \, |\, \| v \| = 1 \}$ and $\minEig{g} = \min \{ \| v \|_{g} \, |\, \| v \| = 1 \}$ respectively, and the \emph{condition number} is defined as $\condN{g} := \maxEig{g} / \minEig{g}$.

In this context, also recall that the 2-norms induced by any two inner products on a finite-dimensional vector space are equivalent, that is, for a vector space $V$ with norms $\| \cdot \|_a$ and $\| \cdot \|_b$ there are constants $\ell>0$ and $L>0$ such that for every $v \in V$ it holds that $\ell \| v \|_a \leq \| v \|_b \leq L \| v \|_a $. For instance, $\ell = \minEig{b} /\maxEig{a}$ and $L = \maxEig{b} / \minEig{a}$.

Hence, we can define a metric as a variable inner product over a given set.

\begin{definition}\label{def:metric}
	Given a set $\calX \subset \bbR^n$, a \emph{(Riemannian) metric} is a map $g: \calX \rightarrow L_2^n$ that assigns to every point $x \in \calX$ an inner product $\left\langle \cdot , \cdot \right\rangle_{g(x)}$. A metric is (Lipschitz) continuous if is  (Lipschitz) continuous as a map from $\calX$ to $L^n_2$.
\end{definition}

If clear from the context at which point $x$ the metric $g$ is applied, we drop the argument in the subscript and write $\left\langle \cdot, \cdot \right\rangle_g$ or $\| \cdot \|_g$. We always retain the subscript $g$, in order to draw a distinction between the Euclidean norm $\| \cdot \|$.

Since $g$ is positive definite for all $x$ by definition, it follows that $\maxEig{g(x)}, \minEig{g(x)}$ and $\condN{g(x)}$ are well-defined for all $x$.
However, $\condN{g(x)}$ is not necessarily locally bounded (even if $g$ is bounded as a map). In particular, $\minEig{g(x)}$ might not be bounded below, away from 0. Hence, for metrics we require the following definition of local boundedness.

\begin{definition}
	A metric $g$ on $\calX$ is \emph{locally weakly bounded} if for every $x \in \calX$ there exist $\ell, L > 0$ such that $\ell \leq \condN{g(y)} \leq L$ holds for all $y \in \calX$ in a neighborhood of~$x$. It is \emph{weakly bounded} if $\ell \leq \condN{g(x)} \leq L$ holds for all $x \in \calX$.
\end{definition}

A metric $g$ can be locally weakly bounded even if its not locally bounded as a map $\calX \rightarrow L^n_2$. Furthermore, since maximum and minimum eigenvalues (and hence the condition number) are continuous functions of a metric (or the representing matrix) it follows that a continuous metric is always locally weakly bounded.

\begin{remark}
	In the following, we will continue to use the Euclidean norm as a distance function on $\bbR^n$ and use any Riemannian metric only in the context of projection directions.
	Thereby, we avoid the notational complexity introduced by Riemannian geometry, and more importantly we do not need to make an a priori assumption on the differentiability on the metric $g$ (which is a prerequisite for many Riemannian constructs to exist), thus preserving a high degree of generality.
\end{remark}

\subsection{Normal Cones} Given a metric $g$, we can define (oblique) normal cones induced by $g$ (see \cref{fig:norm_cone}).

\begin{definition}\label{def:norm_cone} Let $\calX \subset \bbR^n$ be Clarke regular and let $g$ be a metric on $\calX$, then the \emph{normal cone at $x \in \calX$ with respect to $g$} is defined as the polar cone of $T^C_x \calX$ with respect to the metric $g$, i.e.,
	\begin{equation}\label{eq:norm_cone}
		N^g_x \calX := {\left(T_x^C \calX \right)}^* = \left\lbrace \eta \, \middle|\, \forall v \in T^C_x \calX: \, \left\langle v, \eta \right\rangle_{g(x)} \leq 0 \right\rbrace \, .
	\end{equation}
	The normal cone with respect to the Euclidean metric is simply denoted by $N_x \calX$.
\end{definition}

\begin{remark}\label{rem:normal_cones} For simplicity, we will use the notion of normal cone only in the context of Clarke regular sets. If $\calX$ is not Clarke regular, one needs to distinguish between the \emph{regular}, \emph{general} and \emph{Clarke normal cones}~\cite{rockafellarVariationalAnalysis2009}.
\end{remark}

\begin{example}[normal cone to constraint-defined sets]\label{ex:clarke_reg_normal_cone}
	As in \cref{ex:clarke_reg_constraint_set} consider $\calX := \{ x \, | \, h(x) \leq 0\}$ where $h: \bbR^n \rightarrow \bbR^m$ is $C^1$ and $\nabla h(x)$ has full rank for all~$x$. Further, let $g$ denote a metric on $\calX$ represented by $G(x)\in \bbR^{n \times n}$. Then, the normal cone of $\calX$ at $x$ is given by
	\begin{align*}
		N^g_x \calX = \left\lbrace \eta \, \middle| \, \eta = \sum\nolimits_{i \in I(x)} \alpha_i G^{-1}(x) {\nabla h_i(x)}^T, \, \alpha_i \geq 0 \right\rbrace
	\end{align*}
	which can be derived by inserting any $\eta$ into~\eqref{eq:norm_cone} and using $T_x \calX$ in \cref{ex:clarke_reg_constraint_set}.
\end{example}

\section{Projected Dynamical Systems}\label{sec:pds}

With the above notions we can now formally define our main object of study.

\begin{definition}\label{def:proj_vf}
	Given a set $\mathcal X \subset \bbR^n$, a metric $g$ on $\calX$, and a vector field $f: \calX \rightarrow \bbR^n$, the \emph{projected vector field} of $f$ is defined as the set-valued map
	\begin{align}\label{eq:def_proj_vf}
		\tproj{\calX}{g}{f}: \calX \rightrightarrows \bbR^n \qquad
		x \mapsto \underset{v \in T_x \calX}{\arg \min} \| v - f(x) \|^2_{g(x)}
	\end{align}
\end{definition}

For simplicity, we call $\tproj{\calX}{g}{f}$ a \emph{vector field} even though $\tproj{\calX}{g}{f}(x)$ might not be a singleton.
We will write $\tproj{}{}{f}$ whenever $\calX$ and $g$ are clear from the context.

\begin{example}[pointwise evaluation of a projected vector field] As in \cref{ex:clarke_reg_constraint_set,ex:clarke_reg_normal_cone} let $\calX := \{ x \, | \, h(x) \leq 0\}$ where $h: \bbR^n \rightarrow \bbR^m$ is $C^1$ and $\nabla h(x)$ has full rank for all~$x$ and let $g$ denote a metric on $\calX$ represented by $G(x)\in \bbR^{n \times n}$. Furthermore, consider a vector field $f: \calX \rightarrow \bbR^n$. Then, the projected vector field $\tproj{\calX}{g}{f}(x)$ at $x \in \calX$ is given as the solution of the convex quadratic program
	\begin{align*}
		\underset{v \in \bbR^n}{\minimize} \quad {(f(x) - v)}^T G(x) (f(x) - v) \qquad
		\subjto \quad \nabla h_{I(x)}(x) v \leq 0 \, .
	\end{align*}
	Note that $x$ is not an optimization variable. Hence, the properties of $f$ and $g$ as function of $x$ are irrelevant when doing a pointwise evaluation of $\tproj{\calX}{g}{f}(x)$.
\end{example}

Since $T_x \calX$ is non-empty and closed, a minimum norm projection exists, and therefore $\tproj{\calX}{g}{f}(x)$ is non-empty for all $x \in \calX$.\footnote{See, e.g., the first part of the proof of Hilbert's projection theorem~\cite[Prop.~1.37]{peypouquetConvexOptimizationNormed2015}.}
Hence, a \emph{projected dynamical system} is described by the initial value problem
\begin{equation}\label{eq:pds_ivp}
	\dot x \in \tproj{\calX}{g}{f}(x) \,, \qquad x(0) = x_0 \,,
\end{equation}
where $x_0 \in \calX$.
If $T_x \calX$ is convex for all $x$ then $\tproj{\calX}{g}{f}(x)$ is a singleton for all $x \in \calX$ (note that $ \| v - f(x) \|^2_{g(x)}$ is always strictly convex as function of $v$). In this case we will slightly abuse notation and not distinguish between the set-valued map and its induced vector field, i.e., instead of~\eqref{eq:pds_ivp} we simply write $\dot x = \tproj{\calX}{g}{f}(x)$, $x(0) = x_0$.

An absolutely continuous function $x: [0, T) \rightarrow \calX$ with $T>0$ and $x(0) = x_0$ that satisfies $\dot x \in \tproj{\calX}{g}{f}(x)$ almost everywhere (i.e., for all $t \in [0, T)$ except on a subset of Lebesgue measure zero) is called a \emph{Carath\'eodory solution} to~\eqref{eq:pds_ivp}.

\begin{remark}
	The class of systems~\eqref{eq:pds_ivp} can be generalized to $f$ being set-valued, i.e., $f:\bbR^n \rightrightarrows \bbR^n$. This avenue has been explored in~\cite{henryExistenceTheoremClass1973,cornetExistenceSlowSolutions1983,aubinViabilityTheory1991,aubinDifferentialInclusionsSetValued1984}, albeit only for $g$ Euclidean and $\calX$ Clarke regular. In order not to overload our contributions with technicalities we assume that $f$ is single-valued, although an extension is possible.
\end{remark}

As the following example shows, Carath\'eodory solutions to~\eqref{eq:pds_ivp} can fail to exist unless various regularity assumptions $\mathcal X$, $f$ and $g$ hold. Hence, in the next section we propose the use of \emph{Krasovskii solutions} which exist in more general settings. Furthermore, we will show that the Krasovskii solutions reduce to Carath\'eodory solutions under the same assumptions that guarantee the existence of the latter.

\begin{example}[non-existence of Carath\'eodory solution]\label{ex:marble_run1}
	Consider $\bbR^2$ with the Euclidean metric, the uniform ``vertical'' vector field $f = (0,1)$, and the self-similar closed set $\calX$ illustrated in \cref{fig:marble_run} and defined by
	\begin{equation}\label{eq:marble_run_def}
		\calX = \left\lbrace (x_1, x_2) \, \middle|\, \forall k \in \mathbb{Z}: \, x_2 = \pm 2x_1 - \frac{2}{9^k}, |x_2| \leq |x_1| \right\rbrace \cup \{0 \} \, .
	\end{equation}

	\begin{figure}[htb]
		\centering
		\begin{subfigure}[t]{0.25\textwidth}
			\includegraphics[width=.8\textwidth]{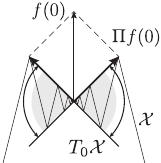}
			\subcaption{}\label{fig:marble_run_proj}
		\end{subfigure}\hspace{.2cm}
		\begin{subfigure}[t]{0.25\textwidth}
			\includegraphics[width=.8\textwidth]{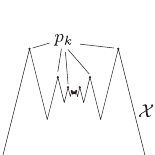}
			\subcaption{}\label{fig:marble_run_equi}
		\end{subfigure}\hspace{.2cm}
		\begin{subfigure}[t]{0.25\textwidth}
			\includegraphics[width=.8\textwidth]{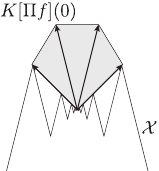}
			\subcaption{}\label{fig:marble_run_krasovskii}
		\end{subfigure}
		\caption{(a) Tangent cone and projected vector field at 0, (b) local equilibria for \cref{ex:marble_run1} and (c) Krasovskii regularization for \cref{ex:marble_run_krasovskii} at 0.}\label{fig:marble_run}
	\end{figure}
	The tangent cone at $0$ is given by $T_0 \calX = \{ (v_1, v_2) \, | \, | v_2 | \leq | v_1 | \}$. It is not ``derivable'', that is, there are no differentiable curves leaving 0 in a tangent direction and remaining in $\calX$. However, by definition there is a sequence of points in $\calX$ approaching~$0$ in the direction of any tangent vector.
	At 0 the projection of $f$ on the tangent cone is not unique as seen in \cref{fig:marble_run_proj}, namely $\tproj{}{}{f}(0) = \left \{\left(\frac{1}{2}, \frac{1}{2}\right), \left(-\frac{1}{2}, \frac{1}{2}\right)\right \}$.

	Furthermore, there is no Carath\'eodory solution to $\dot x \in \tproj{}{}{f}(x)$ for $x(0) = 0$. To see this, we can argue that any solution starting at~0 can neither stay at~0 nor leave~0.
	More precisely, on one hand the constant curve $x(t) = 0$ for $t \in [0, T)$ with $T > 0$ cannot be a solution since it does not satisfy $\dot x \in \tproj{}{}{f}(0)$.
	On the other hand, the points $p_k = \left(\pm \frac{2}{3^{1+2k}}, \frac{2}{3^{1+2k}}\right)$ illustrated in \cref{fig:marble_run_equi} are locally asymptotically stable equilibria of the system. Namely there is an equilibrium point arbitrarily close to $0$. Thus, loosely speaking, any solution leaving $0$ would need to converge to an equilibrium arbitrarily close to~$0$.
\end{example}

\section{Existence of Krasovskii solutions}\label{sec:exist}

The pathology in \cref{ex:marble_run1} can be resolved either by placing additional assumptions on the feasible set $\calX$ or by relaxing the notion of a solution. In this section we focus on the latter.

\begin{definition}\label{def:krasovskii_reg}
	Given a set-valued map $F: \calX \rightrightarrows \bbR^n$, its \emph{Krasovskii regularization} is defined as the set-valued map given by
	\begin{align*}
		\Kras{F}: \calX \rightrightarrows \bbR^n \qquad
		x \mapsto \cocl \underset{y \rightarrow x}{\lim \sup}\, F(y) \, .
	\end{align*}
\end{definition}

Given a set-valued map $F: \calX \rightrightarrows \bbR^n$, an absolutely continuous function $x:[0, T) \rightarrow \calX$ with $T>0$ and $x(0) = x_0$ is a \emph{Krasovskii solution} of the inclusion
\begin{equation*}
	\dot x \in F(x)\, , \qquad x(0) = x_0
\end{equation*}
if it satisfies $\dot x \in \Kras{F}(x)$ almost everywhere. In other words, a Carath\'eodory solution to the regularized set-valued map $\Kras{F}$ is a Krasovskii solution of the original problem.

Hence we can state the following existence result about Krasovskii solutions.

\begin{theorem}[existence of Krasovskii solutions]\label{thm:main_exist}
	Let $\calX \subset \bbR^n$ be a locally compact set, $f:\calX \rightarrow \bbR^n$ a locally bounded vector field and $g$ a locally weakly bounded metric defined on~$\calX$. Then, for any $x_0 \in \calX$ there exists a Krasovskii solution $x:[0, T)\rightarrow \calX$ for some $T>0$ to
	\begin{equation}\label{eq:main_krasovskii_system}
		\dot x \in \tproj{\calX}{g}{f}(x) \qquad x(0) = x_0 \, .
	\end{equation}
	In addition, for $r>0$ such that $U_r := \{ x \in \calX \, | \, \| x - x_0 \| \leq r \}$ is closed and $L = \max_{y \in U_r} \| \Kras{\tproj{\calX}{g}{f}}(y) \|$ exists, the solution is $C^{0,1}$ and exists for $T > r/ L$.
\end{theorem}

\ifARXIV
	\begin{proof}
		We show that the general existence result~\cite[Cor.~1.1]{haddadMonotoneTrajectoriesDifferential1981} (\cref{prop:haddad}) is applicable to Krasovskii regularized projected vector fields. Namely, we need to verify that $\Kras{\tproj{\calX}{g}{f}}$ is convex, compact, non-empty, upper semicontinuous (usc), and
		\begin{equation}\label{eq:haddad_cond}
			\Kras{\tproj{\calX}{g}{f}}(x) \cap T_x \calX \neq \emptyset \qquad \forall x \in \calX \, .
		\end{equation}

		The fact that $\Kras{\tproj{\calX}{g}{f}}$ is closed and convex is immediate from its definition. It is non-empty since $\tproj{\calX}{g}{f}(x)$ is non-empty and $\tproj{\calX}{g}{f}(x) \subset \Kras{\tproj{\calX}{g}{f}}(x)$ for all $x \in \calX$.
		Further, we have $\tproj{\calX}{g}{f}(x) \subset T_x \calX$ by definition for all $x \in \calX$ and therefore~\eqref{eq:haddad_cond} holds.
		For the rest of the proof let $F(x) := {\lim \sup}_{y \rightarrow x} \tproj{\calX}{g}{f}(y)$ (hence, $\Kras{\tproj{\calX}{g}{f}} = \cocl F$).

		Next, we show that $\Kras{\tproj{\calX}{g}{f}}(x)$ is compact for all $x \in \calX$.
		For this, we first introduce an auxiliary metric $\hat g$ defined as $\hat g(x) := g(x) / \maxEig{g(x)}$,
		that is, we scale the metric at every $x \in \calX$ by dividing it by its maximum eigenvalue at that point.
		This implies that $\| f(x) \|_{\hat{g}(x)} \leq \| f(x) \|$ for all $x \in \calX$.
		Note that the projected vector field is unchanged, i.e., $\tproj{\calX}{\hat{g}}{f} = \tproj{\calX}{g}{f}$, since in~\eqref{eq:def_proj_vf} only the objective function is scaled.
		Furthermore, $\condN{g(x)} = \condN{\hat{g}(x)}$ for all $x \in \calX$, and consequently $\hat{g}$ is locally weakly bounded since $g$ is locally weakly bounded.

		Given any $x \in \calX$, since $0 \in T_x \calX$ it follows that $\|v \|_{\hat{g}(x)} \leq \| f(x) - 0 \|_{\hat{g}(x)}$ for every $v \in \tproj{\calX}{\hat{g}}{f}(x)$. Consequently, by local boundedness of $f$ there exists $L''>0$ such that $\| \tproj{\calX}{\hat{g}}{f}(y) \|_{\hat{g}(y)} \leq L''$ for every $y \in \calX$ in a neighborhood of $x$.
		Furthermore, by weak local boundedness of $\hat{g}$ there exists $L' > 0$ such that $\condN{\hat{g}(x)} \leq L'$ in a neighborhood of $x$. Since $\maxEig{\hat{g}(x)} = 1$, it follows that $\minEig{g(x)} \geq 1/L'$ and therefore $ \| v \| \leq L' \| v \|_{g(y)}$ for all $v \in T_y \bbR^n$ and all $y \in \calX$ in a neighborhood of $x$.
		Combining these arguments, there exist $L', L'' > 0$ such that for every $y \in \calX$ in a neighborhood of $x$ it holds that
		\begin{align}\label{eq:proj_bounded}
			\tfrac{1}{L'} \| \tproj{\calX}{\hat{g}}{f}(y) \| \leq \| \tproj{\calX}{\hat{g}}{f}(y) \|_{\hat{g}(y)} \leq \| f(y) \|_{\hat{g}(y)} \leq \| f(y) \| \leq L'' \, .
		\end{align}
		Hence, since $\tproj{\calX}{\hat{g}}{f} = \tproj{\calX}{g}{f}$, it follows that $\tproj{\calX}{g}{f}$ is locally bounded.

		Let $U \subset \calX$ be a compact neighborhood of $x$ such that~\eqref{eq:proj_bounded} holds. Consider the graph of $\tproj{\calX}{g}{f}$ restricted to $U$ given by $\gph \tproj{\calX}{g}{f}|_U := \{ (x, v) \, | \, x \in U, v \in \tproj{\calX}{g}{f}(x) \}$. By definition of the outer limit we have $\cl \gph \tproj{\calX}{g}{f}|_U = \gph F|_U$, i.e., $F$ is the so-called \emph{closure} of $\tproj{\calX}{g}{f}|_U$~\cite[p. 154]{rockafellarVariationalAnalysis2009}. Thus, since $\gph \tproj{\calX}{g}{f}|_U$ is bounded, $\gph F|_U$ is compact, and consequently $F(y)$ is locally bounded for every $y \in U$. In particular, since $F(x)$ is compact, and the closed convex hull of a bounded set is compact~\cite[Thm.~1.4.3]{hiriart-urrutyFundamentalsConvexAnalysis2012}, it follows that $\cocl F(x) = \Kras{\tproj{\calX}{g}{f}}(x)$ is compact for all $x \in \calX$.

		Finally, we need to show that $\Kras{\tproj{\calX}{g}{f}}$ is usc. For this, note that the map $F$ is outer semicontinuous (osc) and closed by definition. Furthermore, it is locally bounded (as shown above). Consequently, by \cref{lem:outer_sem_closedgraph}, $F$ is also usc. Hence, \cref{lem:filippov_convex} states that $\co F$ is usc as well. Since $F(x)$ is compact for all $x \in \calX$, it follows that $\co F(x) = \cocl F(x)$~\cite[Thm.~1.4.3]{hiriart-urrutyFundamentalsConvexAnalysis2012},
		and therefore $\Kras{\tproj{\calX}{g}{f}} = \cocl F$ is usc.

		Thus, $\Kras{\tproj{\calX}{g}{f}}$ satisfies the conditions for \cref{prop:haddad} to be applicable, and therefore the existence of Krasovskii solution to~\eqref{eq:main_krasovskii_system} is guaranteed for all $x_0 \in \calX$.
	\end{proof}
\else
	\cref{thm:main_exist} can be derived from standard viability results, e.g.,~\cite{aubinViabilityTheory1991,goebelHybridDynamicalSystems2012}. The primary technicality is to show that a locally weakly bounded metric results in $\Kras{\tproj{\calX}{g}{f}}$ being locally bounded. For completeness, a self-contained proof can be found in~\cite{hauswirthProjectedDynamicalSystems2018a}.
\fi

Besides weaker requirements for existence, the choice to consider Krasovskii solutions is also motivated by their inherent ``robustness'' towards perturbations, i.e., solutions to a perturbed system still approximate the solutions of the nominal systems~\cite[Chap.~4]{goebelHybridDynamicalSystems2012}. In the same spirit, one can also establish results about the continuous dependence of solutions on initial values and problem parameters~\cite{filippovDifferentialEquationsDiscontinuous1988}.

The existence of solutions for $t\rightarrow \infty$ is guaranteed under the following conditions.

\begin{corollary}[existence of complete solutions]\label{cor:max_sol}
	Consider the same setup as in \cref{thm:main_exist}. If either
	\begin{enumerate}[label = (\roman*)]
		\item\label{enum:glob_exist_1} $\calX$ is closed, $f$ is bounded, and $g$ is weakly bounded, or
		\item\label{enum:glob_exist_2} $\calX$ is compact, $f$ and $g$ are continuous, or
		\item\label{enum:glob_exist_3} $\calX$ is closed, $f$ is globally Lipschitz
					and $g$ is weakly bounded,
	\end{enumerate}
	then for every $x_0 \in \calX$ every Krasovskii solution to \eqref{eq:main_krasovskii_system} can be extended to $T \rightarrow \infty$.
\end{corollary}

\ifARXIV
	\begin{proof}
		\ref{enum:glob_exist_1} If $f$ is bounded and $g$ is weakly bounded, then the local boundedness argument of the proof of \cref{thm:main_exist} can be applied globally, i.e.,~\eqref{eq:proj_bounded} holds for all $y \in \calX$ for the same $L', L''$ and hence $\Kras{\tproj{\calX}{g}{f}}$ is bounded. Hence, in \cref{thm:main_exist} the constant $L > 0$ exists for $r \rightarrow \infty$ and consequently $T \rightarrow \infty$.

		\ref{enum:glob_exist_2} Since $f$ is continuous it only takes bounded values on a compact set. Furthermore, continuity of $g$ implies local weak boundedness, i.e., for every $x \in \calX$ there exist $\ell_x, L_x > 0$ such that $\ell_x < \condN{g(y)} < L_x$ for all $y \in \calX$ in a neighborhood of $x$. Since $\calX$ is compact, there exist $\ell := \min_{x \in \calX} \ell_x$ and $L := \max_{x \in \calX} L_x$ and~\eqref{eq:proj_bounded} holds for all $y \in \calX$. Hence, $g$ is weakly bounded. Then, the same arguments as for \cref{enum:glob_exist_1} apply.

		\ref{enum:glob_exist_3} Assume without loss of generality that $0 \in \calX$ (possibly after a linear translation). Global Lipschitz continuity of $f$ implies the existence of $L'' > 0$ such that $\| f(x) \| \leq L'' ( \| x \| + 1)$ for all $x \in \calX$ (\emph{linear growth} property~\cite{aubinViabilityTheory1991}). To see this, recall that by the reverse triangle inequality and the definition of Lipschitz continuity there exists $L' > $ such that $| \| f(x) \| - \| f(0) \| | \leq \| f(x) - f(0) \| \leq L' \| x \|$ for all $x, y \in \calX$. It follows that $\| f(x) \| \leq L' \| x \| + \|f(0)\|$ and hence $L''$ can be chosen as the maximum of $L'$ and $\|f(0)\|$ to yield the linear growth property.

		Since $g$ is weakly bounded, the same arguments used for~\eqref{eq:proj_bounded} can be used to establish that there exists $L'''> 0$ such that for all $x \in \calX$ it holds that
		\begin{align*}\label{eq:proj_bounded2}
			L''' \| \tproj{\calX}{g}{f}(x) \| < \| \tproj{\calX}{g}{f}(x) \|_{g(x)} \leq \| f(x) \|_{g(x)} \leq \| f(x) \| < L'' (\|x \| + 1) \, .
		\end{align*}
		It follows by the same arguments as in the proof of \cref{thm:main_exist} that $\| \Kras{\tproj{\calX}{g}{f}}(x) \| \leq L (\|x \| + 1)$ where $L = L''/L'''$, i.e., the linear growth condition applies to $\Kras{\tproj{\calX}{g}{f}}$.

		Hence using standard bounds~\cite[p. 100]{aubinViabilityTheory1991}, one can conclude that any Krasovskii solution to~\eqref{eq:main_krasovskii_system} satisfies $\| x (t) \| \leq (\|x_0 \| + 1) e^{L t}$. Namely, define $u(t) := L( \| x(t) \| + 1)$ and note that $\dot u(t) = L \frac{d}{dt} \| x(t) \| =  L \langle x(t) / \|x(t) \|, \dot x(t) \rangle \leq L \| \dot x(t) \| \leq L^2 ( \| x(t) \| + 1) = L u(t)$ holds for all $t$ where $\dot x(t)$ exists. Hence, Gronwall's inequality (for discontinuous ODEs) implies the desired bound. It immediately follows that $x(t)$ cannot have finite escape time and therefore can be extended to $t \rightarrow \infty$, completing the proof of \cref{enum:glob_exist_3}.
	\end{proof}
\else
	The proof of \cref{cor:max_sol} is standard and can be found in~\cite{hauswirthProjectedDynamicalSystems2018a}.  For instance, \cref{enum:glob_exist_3} requires a Gronwall-argument to preclude finite escape times.
\fi

\begin{example}[existence of Krasovskii solutions]\label{ex:marble_run_krasovskii}
	Consider again the setup of \cref{ex:marble_run1}. The Krasovskii regularization at $0$ of the projected vector field $\tproj{}{}{f}$ is shown in \cref{fig:marble_run_krasovskii}. It is the convex hull of five limiting vectors: the two vectors in $\tproj{}{}{f}(0)$, the projected vector field at the arbitrarily close-by equilibria $p_k$ which is $\tproj{}{}{f}(p_k) = 0$ and the projected vectors at the ascending and descending slopes.

	Note that the map $x(t) = 0$ for all $t \geq 0$ is a valid solution to the differential inclusion $\dot x \in \Kras{\tproj{}{}{f}}(x)$ with initial point $0$ and hence a Krasovskii solution to the projected dynamical system, but not a Carath\'eodory solution.
\end{example}

\subsection{Additional Lemmas}

For future reference we state the following two key lemmas about projected vector fields and their Krasovskii regularizations.
\unless\ifARXIV
	Proofs for both results are simple but tedious and can be found in~\cite{hauswirthProjectedDynamicalSystems2018a}. They both rely on Moreau's Decomposition Theorem~\cite[Thm.~3.2.5]{hiriart-urrutyFundamentalsConvexAnalysis2012} and generalize results in~\cite{cornetExistenceSlowSolutions1983} to the case of a variable metric and Krasovskii-regularized maps.
\fi

\begin{lemma}\label{lem:moreau_gen}
	Given $\mathcal X$, $g$, and $f$ as in \cref{def:proj_vf}, for any $v \in \tproj{\calX}{g}{f}(x)$ one has $\left\langle f(x), v \right\rangle_{g(x)} = \| v \|^2_{g(x)}$. If in addition $\calX$ is Clarke regular at $x$, then $\tproj{\calX}{g}{f}(x)$ is a singleton and there is $\hat{\eta} \in N_x^g \calX$ such that the following equivalent statements hold:
	\begin{enumerate} [label = (\roman*)]
		\item\label{moreau1} $\tproj{\calX}{g}{f}(x) = f(x) - \hat{\eta}$,
		\item\label{moreau2} $\arg {\min}_{\eta \in N^g_x \calX} \| \eta - f(x)\|_{g(x)} = \hat{\eta}$,
		\item $f(x) - \hat{\eta} \in T_x \calX$ and $\left\langle x - \hat{\eta}, \hat{\eta} \right\rangle_{g(x)} = 0$.
	\end{enumerate}
\end{lemma}

\ifARXIV
	\begin{proof}
		Let $v \in \tproj{\calX}{g}{f}(x)$. As $T_x \calX$ is a cone we have $\lambda v \in T_x \calX$ for all $\lambda \geq 0$. Since $v$ (locally) minimizes $\| v - f(x) \|_{g(x)}^2$ over $T_x \calX$, it follows that $\lambda = 1$ minimizes $M(\lambda) := \tfrac{1}{2}\| \lambda v - f(x) \|_{g(x)}^2$ for $v$ fixed. Hence, for $\lambda = 1$ the optimality condition $\tfrac{dM}{d\lambda}(\lambda)  = \lambda \left\langle v - f(x), v \right\rangle_{g(x)} = 0$ holds. This proves the first part.
		The second part follows from Moreau's Theorem~\cite[Thm.~3.2.5]{hiriart-urrutyFundamentalsConvexAnalysis2012} since $T_x \calX$ is convex by Clarke regularity.
	\end{proof}
\else

\fi

\begin{lemma}\label{lem:kras_normal} Consider $\mathcal X \subset \bbR^n$, let $g$ be a continuous metric on $\calX$ and $f$ a continuous vector field on $\calX$. Then, for every $v \in \Kras{\tproj{\calX}{g}{f}}(x)$, one has $\left\langle f(x), v \right\rangle_{g(x)} \geq \| v \|_{g(x)}^2$. If in addition $\calX$ is Clarke regular, then for $\hat{\eta} := f(x) - v$ we have $\hat{\eta} \in N^g_x \calX$.
\end{lemma}

\ifARXIV
	\begin{proof} Let $F(x) := {\lim \sup}_{y \rightarrow x} \tproj{\calX}{g}{f}(y)$. By definition of the outer limit, there exist sequences $x_k \rightarrow x$ with $x_k \in \calX$ and $v_k \rightarrow v$ with $v_k \in \tproj{\calX}{g}{f}(x_k)$ for every $v \in F(x)$ and every $x \in \calX$.
		In particular, $\left\langle f(x_k), v_k \right\rangle_{g(x_k)}= \| v_k \|_{g(x_k)}^2$ holds for every $k$ by \cref{lem:moreau_gen}. Since $f$ and $g$ are continuous the equality holds in the limit, i.e., $\left\langle f(x), v \right\rangle_{g(x)} = \| v \|^2_{g(x)}$ for every $v \in F(x)$. Taking any convex combination $v = \sum_i \alpha_i v_i$ with $v_i \in F(x)$ and $\alpha_i \geq 0$ and $\sum_i \alpha_i = 1$, we have
		\begin{equation*}
			\sum\nolimits_i \left\langle f(x), \alpha_i v_i  \right\rangle_{g(x)} =
			\sum\nolimits_i \alpha_i \| v_i \|^2_{g(x)} \geq
			{\left \| \sum \nolimits_i \alpha_i v_i \right \|}_{g(x)}^2 = \left \| v \right \|^2_{g(x)} \, ,
		\end{equation*}
		and therefore $\left\langle f(x), v \right\rangle_{g(x)} \geq \| v \|^2_{g(x)}$ for every $v \in \cocl F(x) = \Kras{\tproj{\calX}{g}{f}}(x)$.

		According to \cref{lem:moreau_gen}, if $\calX$ is Clarke regular, given a sequence $x_k \rightarrow x$, the sequences $v_k = \tproj{\calX}{g}{f}(x_k)$ and $\hat{\eta}_k \in N_{x_k}^g \calX$ for which $\hat{\eta}_k = f(x_k) - \tproj{\calX}{g}{f}(x_k)$ are uniquely defined. Since $g$ is continuous, the mapping $x \mapsto N^g_x \calX$ is outer semi-continuous (\cref{lem:normal_outer_semi}) and therefore $\lim_{k \rightarrow \infty} \hat{\eta}_k \in N^g_x \calX$. In other words, for every $v \in F(x)$ it holds that $f(x) - v \in N_x^g \calX$. Since by Clarke regularity $N_x^g\calX$ is convex, it follows that, for any convex combination $\eta = \sum_i \alpha_i (f(x) - v_i)$ with $v_i \in F(x)$ and $\alpha_i \geq 0$ and $\sum_i \alpha_i = 1$, it must hold that $\eta \in N_x^g \calX$, which completes the proof.
	\end{proof}
\else

\fi

\section{Equivalence of Krasovskii and Carath\'eodory Solutions}\label{sec:equiv}

In this section we study the relation between Carath\'eodory and Krasovskii solutions. In particular, we show that the solutions are equivalent if the metric is continuous and the feasible domain is Clarke regular, thus recovering (for the Euclidean metric) known existence conditions for Carath\'eodory solutions. Further, we establish the connection to related work~\cite{aubinViabilityTheory1991,aubinDifferentialInclusionsSetValued1984,cornetExistenceSlowSolutions1983}.

\begin{definition} Consider a set $\calX \subset \bbR^n$, a metric $g$ and a vector field $f$, both defined on $\calX$. The \emph{sets of Carath\'eodory} and \emph{Krasovskii solutions} of~\eqref{eq:pds_ivp} with initial condition $x_0 \in \calX$ are respectively given by
	\begin{align*}
		\mathcal{S}_C(x_0) & : = \left\lbrace x \middle| x: [0, T) \rightarrow \calX, \, T > 0, \, x \in C^A, \, x(0) = x_0, \, \dot x(t) \in \tproj{\calX}{g}{f}(x(t))\text{a.e.} \right\rbrace         \\
		\mathcal{S}_K(x_0) & : = \left\lbrace x \middle| x: [0, T) \rightarrow \calX, \, T > 0, \, x \in C^A, \, x(0) = x_0, \, \dot x(t) \in \Kras{\tproj{\calX}{g}{f}}(x(t)) \mbox{a.e.} \right\rbrace
	\end{align*}
	where \emph{a.e.} means \emph{almost everywhere} and $C^A$ denotes absolutely continuous functions.
\end{definition}

Since $\tproj{\calX}{g}{f}(x) \subset \Kras{\tproj{\calX}{g}{f}}(x)$, it is clear that every Carath\'eodory solution of~\eqref{eq:pds_ivp} is also a Krasovskii solution, i.e., $\mathcal{S}_C(x_0) \subset \mathcal{S}_K(x_0)$ for all $x_0 \in \calX$. A pointwise condition for the equivalence of the solution sets is given as follows:

\begin{lemma}\label{lem:equiv}
	Given any set $\calX$, metric $g$ and vector field $f$, if $\Kras{\tproj{\calX}{g}{f}}(x) \cap T_x \calX = \tproj{\calX}{g}{f}(x)$ holds for all $x \in \calX$, then $\mathcal{S}_C(x_0) = \mathcal{S}_K(x_0)$ for all $x_0 \in \calX$.
\end{lemma}

\begin{proof}
	Since, $\mathcal{S}_C(x_0)~\subset~\mathcal{S}_K(x_0)$, we only need to consider $x \in \mathcal{S}_K(x_0)$ and show that $x \in \mathcal{S}_C(x_0)$. By \cref{lem:tgt_deriv}, $\dot x(t) \in T_{x(t)}\calX$ holds for $x(t)$ almost everywhere. Consequently, $\dot x(t) \in \Kras{\tproj{\calX}{g}{f}(x}(t)) \cap T_{x(t)} \calX$ almost everywhere, and therefore, by assumption, $\dot x(t) \in \tproj{\calX}{g}{ f(x}(t))$.
\end{proof}

The proof of the next result follows ideas from~\cite{cornetExistenceSlowSolutions1983}. The requirement that $g$ and $f$ need to be continuous deserves particular attention.

\begin{theorem}[equivalence of solution sets]\label{thm:main_equiv}
	If $\calX$ is Clarke regular, $g$ is a continuous metric on $\calX$, and $f$ is continuous on $\calX $, then $ \mathcal{S}_C (x_0) = \mathcal{S}_K (x_0)$ for all $x_0 \in \calX$.
\end{theorem}

\begin{proof} It suffices to show that under the proposed assumptions \cref{lem:equiv} is applicable. By definition of $\tproj{\calX}{g}{f}(x)$ we have $\tproj{\calX}{g}{f}(x) \subset \Kras{\tproj{\calX}{g}{f}}(x) \cap T_x \calX $. For the converse, let $ v \in \Kras{\tproj{\calX}{g}{f}}(x) \cap T_x \calX$.
	By \cref{lem:kras_normal}, $v = f(x) - \hat{\eta}$ for some $\hat{\eta} \in N^g_x \calX$ and $\| v \|^2_{g(x)} \leq \left\langle v, f(x) \right\rangle_{g(x)}$.
	Since $\left\langle v, \eta \right\rangle_{g(x)} \leq 0$ for all $\eta \in N^g_x \calX$ we have
	\begin{equation*}
		\| v \|_{g(x)}^2 \leq \left\langle v, f(x) \right\rangle_{g(x)} - \left\langle v, \eta \right\rangle_{g(x)} \leq \| v \|_{g(x)} \| f(x) - \eta \|_{g(x)} \qquad \forall \eta \in N^g_x \calX \, ,
	\end{equation*}
	where the second inequality is due to Cauchy-Schwarz, and therefore $\| v - \hat{\eta}\|_{g(x)} \leq \| f(x) - \eta \|_{g(x)}$ holds for all $\eta \in N^g_x \calX$.
	However, according to \cref{lem:moreau_gen} the fact that $\hat{\eta} = \arg \underset{\eta \in N^g_x \calX}{\min} \| f(x) - \eta \|_{g(x)}$ is equivalent to $v \in \tproj{\calX}{g}{f }(x)$.
\end{proof}

Note that \cref{ex:marble_run1,ex:marble_run_krasovskii} show a case where the conclusion of \cref{thm:main_equiv} fails to hold because $\calX$ is not Clarke regular at the origin. Hence, our sufficient characterization in terms of Clarke regularity is also a sharp one.

\cref{thm:main_equiv} also serves as an existence result of Carath\'eodory solutions, that recovers the conditions derived in~\cite{cornetExistenceSlowSolutions1983}, but for a general metric.

\begin{corollary}[Existence of Carath\'eodory solutions]\label{cor:cara_exist}
	If $\calX$ is Clarke regular, and $g$ and $f$ are continuous on $\calX$, then there exists a Carath\'eodory solution $x:[0, T) \rightarrow \calX$ of~\eqref{eq:pds_ivp} with $x(0) = x_0$ for some $T> 0$, and every $x_0 \in \calX$.
\end{corollary}

Uniqueness, however, requires additional assumptions as will be shown in \cref{sec:uniq}. In particular, uniqueness of the projection $\tproj{\calX}{g}{f}(x)$ does not imply uniqueness of the trajectory (see forthcoming \cref{rem:unique_proj}).

\subsection{Related work and alternative formulations}

With the statements of \cref{sec:equiv} at hand, we discuss their connection to related literature.
As discussed in the introduction, projected dynamical system have been studied from different perspectives and with various applications in mind. In particular, a number of alternative, but equivalent formulations do exist~\cite{brogliatoEquivalenceComplementaritySystems2006,heemelsProjectedDynamicalSystems2000}, but none considers the case of a variable metric. In the following, we discuss a well-established formulation~\cite{aubinDifferentialInclusionsSetValued1984,aubinViabilityTheory1991,cornetExistenceSlowSolutions1983} that has a number of insightful properties.

Namely, under Clarke regularity of the feasible set $\calX$ we may define an alternative differential inclusion given by the initial value problem
\begin{equation}\label{eq:norm_incl}
	\dot x \in f(x) - N^g_x \calX \,, \qquad x(0) = x_0 \in \calX
\end{equation}
and define the solution set as
\begin{equation*}
	\mathcal{S}_N(x_0) : = \left\lbrace x \,\middle|\, x: [0, T) \rightarrow \calX, \, T > 0, \, x \in C^A, \, x(0) = x_0, \, \dot x \in f(x) - N^g_x \calX \text{ a.e.}\right\rbrace \, .
\end{equation*}

The next result is an adaptation of~\cite[Thm.~2.3]{cornetExistenceSlowSolutions1983} to arbitrary metrics. We provide a self-contained proof for completeness.

\begin{corollary}\label{cor:equiv_normal}
	Consider a Clarke regular set $\calX \subset \bbR^n$, a continuous vector field $f$, and a continuous metric $g$, both defined on $\calX$. Then, $\mathcal{S}_N(x_0) = \mathcal{S}_C(x_0)$ holds for systems of the form~\eqref{eq:pds_ivp} and \eqref{eq:norm_incl}, and for all $ x_0 \in \calX$.
\end{corollary}

In short, any solution to~\eqref{eq:norm_incl} is a Carath\'eodory solution of~\eqref{eq:pds_ivp} and vice versa.
However, \cref{cor:equiv_normal} makes no statement about existence of solutions. In fact, the non-compactness of $N_x^g \calX$ prevents us from applying the same viability result as for \cref{thm:main_exist}.

\begin{proof}
	We first note that $\mathcal{S}_C(x_0) \subset \mathcal{S}_N(x_0)$ since $\tproj{\calX}{g}{f}(x) \subset f(x) - N^g_x \calX$ for all $x \in \calX$ by virtue of \cref{lem:kras_normal} and since $\calX$ is Clarke regular.
	Conversely, let $x \in \mathcal{S}_N(x_0)$ be defined for $t \in [0, T)$ for $T > 0$. Then for almost all $t$, we have $\dot x(t) \in f(x(t)) - N^g_{x(t)}\calX$ and $\dot x(t) \in T_{x(t)} \calX \cap -T_{x(t)} \calX$ by \cref{lem:tgt_deriv}. Thus, for $\dot x(t) = f(x(t)) - \eta(x(t))$ with $\eta(x(t)) \in N^g_{x(t)}\calX$ it must hold that
	\begin{equation*}
		\left\langle f(x(t)) - \eta(x(t)), \eta(x(t)) \right\rangle_{g(x(t))} \leq 0 \quad \mbox{and} \quad \left\langle f(x(t)) - \eta(x(t)), - \eta(x(t)) \right\rangle_{g(x(t))} \leq 0 \, .
	\end{equation*}
	Consequently, $\left\langle f(x(t)) - \eta(x(t)), \eta(x(t)) \right\rangle_{g(x(t))} = 0$, and using \cref{lem:moreau_gen} it follows that $\dot x(t) = \tproj{\calX}{g}{ f(x}(t))$.
\end{proof}

\begin{remark}
	Defining inclusions of the form~\eqref{eq:norm_incl} for a set $\calX$ that is not Clarke regular is possible but technical since one would need to distinguish between different types of normal cones (\cref{rem:normal_cones}). Furthermore, depending on the choice of normal cone the resulting set of solutions can be overly relaxed or too restrictive.
\end{remark}

\begin{remark}
	Using \cref{moreau2} in \cref{lem:moreau_gen} it follows that whenever $\dot x$ exists, we have $\dot x = \arg \min_{v \in f(x) - N_x^g \calX} \| v \|_{g(x)}$. When $g$ is the Euclidean metric, this \emph{minimum norm} property gives rise to so-called \emph{slow} solutions of~\eqref{eq:norm_incl}~\cite[Chap.~10.1]{aubinViabilityTheory1991}. For a general metric, the definition of a slow solution generalizes accordingly. However, the property of being ``slow'' depends on the metric.
\end{remark}

\section{Prox-regularity and Uniqueness of Solutions}\label{sec:uniq}

Next, we introduce a generalized definition of \emph{prox-regular} sets on non-Euclidean spaces  with a variable metric and show their significance for the uniqueness for solutions of projected dynamical systems. In the Euclidean setting prox-regularity is well-known to be a sufficient condition on the feasible domain $\calX$ for uniqueness~\cite{cornetExistenceSlowSolutions1983}.

The key issue of this section is thus to generalize the definition of prox-regular sets  that can be used on low-regularity Riemannian manifolds. Previously, prox-regularity has been defined and studied on \emph{smooth} (i.e., $C^\infty$) Riemannian manifolds in~\cite{hosseiniMetricProjectionProxregular2013,bernicotSweepingProcessProxregular2015} using standard \emph{geodesic} notions from Riemannian geometry. In this paper, we weaken the smoothness assumption but, consequently, we cannot apply to the same toolset that requires the existence of unique geodesics (which is only guaranteed on sufficiently smooth manifolds~\cite{hartmanLocalUniquenessGeodesics1950}). Instead we pursue a more low-level approach which the novel insights prox-regularity of a set is independent of the choice of metric (and, more precisely, preserved under $C^{1,1}$ coordinate transformations). This feature is particularly important for envisioned applications in optimization where the feasible domain is given, but choice of metric is often a design parameter of an algorithm.

\subsection{Prox-regularity on non-Euclidean spaces}
For illustration, we first recall and discuss the definition of prox-regularity in Euclidean space. Our treatment of the topic is deliberately kept limited. For a more general overview see~\cite{adlyPreservationProxRegularitySets2016,poliquinProxregularFunctionsVariational1996}.

\begin{definition}\label{def:prox_reg_trad} A Clarke regular set $\calX \subset \bbR^n$ is \emph{prox-regular at $x \in \calX$} if there is $L > 0$ such that for every $z,y \in \calX$ in a neighborhood of $x$ and $\eta \in N_y \calX$ we have
	\begin{align}\label{eq:prox_reg_trad}
		\left\langle \eta, z - y \right\rangle \leq L \| \eta \| \| z - y \|^2 \, .
	\end{align}
	The set $\calX$ is \emph{prox-regular} if it is prox-regular at every $x \in \calX$.
\end{definition}

One of the key features of a prox-regular set $\calX$ is that for every point in a neighborhood of $\calX$ there exists a unique projection on the set~\cite[Def.~2.1, Thm.~2.2]{adlyPreservationProxRegularitySets2016}.

\begin{example}[Prox-regularity in Euclidean spaces]\label{ex:prox_reg_basic}
	Consider the parametric set
	\begin{equation}
		\calX_\alpha := \left\lbrace (x_1, x_2) \, \middle | \, |x_2| \geq {\max \{0, x_1 \}}^\alpha\right\rbrace\label{eq:prox_reg_example}
	\end{equation}
	where $0 < \alpha < 1$ and which is illustrated in \cref{fig:prox_reg}. For $\alpha \leq 0.5$ the set is prox-regular everywhere. In particular for the origin, a ball with non-zero radius can be placed tangentially such that it only intersects the set at 0. For $\alpha > 0.5$ on the other hand the set is not prox-regular at the origin. In fact, all points on the positive axis have a non-unique projection on $\calX_\alpha$ as illustrated in \cref{fig:non-uniq_proj}.
	\begin{figure}[htb]
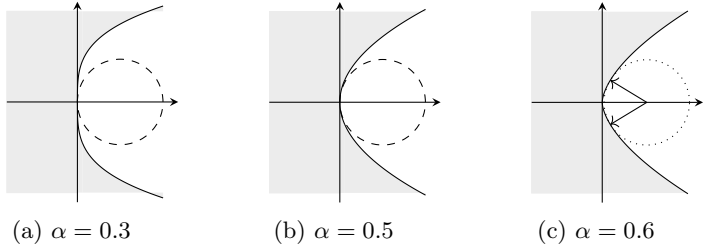

		\vspace{.7cm}
		\centering
		\begin{subfigure}[t]{0.20\textwidth}
			\centering
			\watershedprox{0.3}{0.7}
			\subcaption{$\alpha = 0.3$}
		\end{subfigure} \hspace{.05\textwidth}
		\begin{subfigure}[t]{0.20\textwidth}
			\centering
			\watershedprox{0.5}{0.7}
			\subcaption{$\alpha = 0.5$}
		\end{subfigure} \hspace{.05\textwidth}
		\begin{subfigure}[t]{0.20\textwidth}
			\centering
			\watershedprox{0.6}{0.7}
			\subcaption{$\alpha = 0.6$}\label{fig:non-uniq_proj}
		\end{subfigure}
		\caption{Set $\calX_\alpha$ for different $\alpha$. In (a) and (b) the set $\calX_\alpha$ is prox-regular, unlike in (c).}\label{fig:prox_reg}
	\end{figure}
\end{example}

\cref{def:prox_reg_trad} cannot be directly generalized to non-Euclidean spaces since it requires the distance $\| y - x \|$ between two points in $\calX$. Hence, in~\cite{hosseiniMetricProjectionProxregular2013,bernicotSweepingProcessProxregular2015} prox-regularity is defined on smooth (i.e., $C^{\infty}$) Riemannian manifolds resorting to geodesic distances.
For our purposes we can avoid the notational complexity of Riemannian geometry, yet preserve a higher degree of generality. Thus, we introduce the following definitions.

\begin{definition}\label{def:prox_normal}
	Given a Clarke regular set $\calX \subset \bbR^n$ and a metric $g$, a normal vector $\eta \in N^g_x \calX$ at $ x \in \calX$ is \emph{$L$-proximal with respect to $g$} for $L \geq 0$ if for all $y \in \calX$ in a neighborhood of $x$ we have
	\begin{align}\label{eq:prox_reg_new}
		\left\langle \eta, y - x \right\rangle_{g(x)} \leq L \| \eta \|_{g(x)} \| y - x \|^2_{g(x)} \, .
	\end{align}
	The cone of all $L$-proximal normal vectors at $x$ with respect to $g$ is denoted by $\bar{N}^{g,L}_x \calX$.
\end{definition}

A crucial detail in~\eqref{eq:prox_reg_new} is the fact that $g$ is evaluated at $x$ and is used as an inner product on $\bbR^n$ (which is a slight abuse of notation). In other words, we exploit the canonical isomorphism between $\bbR^n$ and $T_x \bbR^n$ to use $g(x)$ as an inner product on $\bbR^n$.

\begin{definition}\label{def:prox_reg_new}
	A Clarke regular set $\calX \subset \bbR^n$ with a metric $g$ is \emph{$L$-prox-regular at $x \in \calX$ with respect to $g$} if $\bar{N}_y^{g,L} \calX = N^g_y \calX$ for all $y \in \calX$ in a neighborhood of $x$.
	The set $\calX$ is \emph{prox-regular with respect to $g$} if for every $x \in \calX$ there exists $L > 0$ such that $\calX$ is $L$-prox-regular at $x$ with respect to $g$.
\end{definition}

\begin{remark}
	Note that if $g$ is the Euclidean metric, \cref{def:prox_reg_new} reduces to \cref{def:prox_reg_trad}.
	Moreover, when applied to a smooth Riemannian manifold, \cref{def:prox_reg_new} reduces to the definition of prox-regularity given in \cite{hosseiniMetricProjectionProxregular2013,bernicotSweepingProcessProxregular2015}.
	To see this, consider a closed subset $\calX$ of a (geodesically complete) smooth Riemannian manifold $\calM$ with metric $g$. In \cite{hosseiniMetricProjectionProxregular2013,bernicotSweepingProcessProxregular2015}, the $L$-proximal normal cone of $\calX$ at $x \in \calX$ is defined as the set of all $\eta \in T_x \calM$ such that
	\begin{align*}
		\left\langle \eta, \exp^{-1}_x(y) \right\rangle_{g(x)} \leq L \| \eta \| \left \| \exp^{-1}_x(y) \right\|^2_{g(x)}
	\end{align*}
	holds for all $y \in \calM$ in a neighborhood of $x$ and $\exp^{-1}_x(y)$ is the inverse of the exponential map. Namely, $\exp^{-1}_{x}(y)$ maps $y$ to a tangent vector $w \in T_x \calM$ at $x$ such that the geodesic segment between $x$ and $y$ starting from $x$ in the direction $w$ has length $\| w \|_{g(x)}$. With this local bijection between $T_x \calM$ and $\calM$, prox-regularity of $\calX$ can be defined similarly to \cref{def:prox_reg_new}, albeit smoothness and geodesic completeness of $\calM$ (as well as other technical assumptions, e.g., \cite[Ass. 2.9]{bernicotSweepingProcessProxregular2015}) are a prerequisite.
\end{remark}

The following result shows that prox-regularity is in fact independent of the metric. This is the first step towards a coordinate-free definition of prox-regularity.

\begin{proposition}\label{prop:prox_invariance}
	Let $\calX \subset \bbR^n$ be Clarke regular. If $\calX$ is prox-regular with respect to a $C^{0}$ metric $g$, then it is prox-regular with respect to any other $C^0$ metric.
\end{proposition}

In particular if $\calX$ is prox-regular with respect to the Euclidean metric, i.e., according to \cref{def:prox_reg_trad}, then it is prox-regular in any other continuous metric on $\bbR^n$. For the proof of \cref{prop:prox_invariance} we require the following lemma.

\begin{lemma}\label{lem:prox_invar_local}
	Let $\calX \subset \bbR^n$ be Clarke regular and consider to metrics $g, g'$ defined on $\calX$. If for $x \in \calX$ there is $L > 0$ such that $\bar{N}_x^{g,L} \calX = N^g_x \calX$ then $\bar{N}_x^{g',L'} \calX = N^{g'}_x \calX$ holds for $L' \geq \condN{g(x)} \condN{g'(x)} L$.
\end{lemma}

\begin{proof}
	First note that for every $x \in \calX$ the two metrics $g$ and $g'$ induce a bijection between $N^g_x \calX$ and $N^{g'}_x \calX$. Namely, we define $q: T_x \bbR^n \rightarrow T_x\bbR^n$ as the unique element $q(v)$ that satisfies by $\left\langle v, w \right\rangle_{g(x)} = \left\langle q(v), w \right\rangle_{g'(x)}$ for all $w \in T_x\bbR^n$.
	To clarify, in matrix notation we can write $ v^T G(x) w = {q(v)}^T G'(x) w$ and since $G(x), G'(x)$ are symmetric positive definite we have $q(v) := {G'(x)}^{-1} G(x) v$.
	It follows that if $\eta \in N_x^g \calX$ (hence, by definition $\left\langle \eta, w \right\rangle_{g(x)} \leq 0$ for all $w \in T_x \bbR^n$), then $q(\eta) \in N_x^{g'}\calX$.
	Furthermore, omitting the argument $x$, we have $\| q(\eta)\|_{g'} = \eta^T G {G'}^{-1} G \eta \geq 1/\maxEig{g'} \| G \eta \| $ and $\| \eta\|_{g} = \eta^T G {G}^{-1} G \eta \leq 1/\minEig{g} \| G \eta \|$, and therefore $\| q(\eta)\|_{g'(x)} \geq \minEig{g(x)} /\maxEig{g'(x)} \| \eta \|_{g(x)}$.

	Hence, let $\eta \in N^g_x \calX \setminus \{ 0 \}$ be a $L$-proximal normal vector, then
	\begin{equation*}
		\left\langle \tfrac{q(\eta)}{\| q(\eta) \|_{g'(x)}}, y - x \right\rangle_{g'(x)}
		\leq \tfrac{\maxEig{g'(x)}}{\minEig{g(x)}} \left\langle \tfrac{\eta}{\| \eta \|_{g(x)}}, y - x \right\rangle_{g(x)}
		\leq \tfrac{\maxEig{g'(x)}}{\minEig{g(x)}}L \| y - x \|^2_{g(x)} \, .
	\end{equation*}
	Finally, using the equivalence of norms, we have
	\begin{equation}\label{eq:prox_reg_local}
		\tfrac{\maxEig{g'(x)}}{\minEig{g(x)}}L \| y - x \|^2_{g(x)}
		\leq \tfrac{\maxEig{g'(x)}}{\minEig{g(x)}} \tfrac{\maxEig{g(x)}}{\minEig{g'(x)}} L \| y - x \|^2_{g'(x)}
		\leq L' \| y - x\|^2_{g'(x)} \, ,
	\end{equation}
	where $L' \geq \condN{g(x)} \condN{g'(x)} L$. Thus, we have shown that if $v \in \bar{N}_x^{g,L} \calX = N^{g}_x \calX$ then $q(v) \in  \bar{N}_x^{g',L'} = N^{g'}_x \calX$ which completes the proof.
\end{proof}

\begin{proof}[Proof of \cref{prop:prox_invariance}]
	Since $g$ and $g'$ are continuous it follows that $\condN{g(x)}$ and $\condN{g'(x)}$ are continuous in $x$ and therefore locally bounded. Given any $x \in \calX$ and using the pointwise result in \cref{lem:prox_invar_local}, we can choose $L'> 0$ such that~\eqref{eq:prox_reg_local} is satisfied for all $y \in \calX$ in a neighborhood of $x$.
\end{proof}

We conclude this section by showing that feasible domains defined by $C^{1,1}$ constraint functions are prox-regular under the usual constraint qualifications.

\begin{example}[prox-regularity of constraint-defined sets]\label{ex:clarke_reg_prox_cone}
	As in \cref{ex:clarke_reg_constraint_set,ex:clarke_reg_normal_cone} let $h: \bbR^n \rightarrow \bbR^m$ be $C^1$ and $\nabla h(x)$ have full rank for all $x$ and consider $\calX := \{ x \, | \, h(x) \leq 0 \}$. If in addition, $h$ is a $C^{1,1}$ map, then $\calX := \{ x \, | \, h(x) \leq 0 \}$ is prox-regular with respect to any $C^0$ metric $g$ on $\bbR^n$.

	To see this, we consider the Euclidean case without loss of generality as a consequence of \cref{prop:prox_invariance}. We first analyze the sets $\calX_i := \{ x \, | \, h_i(x) \leq 0 \}$ and then show prox-regularity of their intersection. For this, we only need to consider points $x \in \partial \calX_i$ on the boundary of $\calX_i$ since for all $\bar{x} \notin \partial \calX_i$ we have $N_{\bar{x}} \calX_i = \{ 0 \}$ and prox-regularity is trivially satisfied. Hence, using the Descent \cref{lem:c11_lipschitz}, for all $z,y \in \bbR^n$ in a neighborhood of $x$ and all $i=1, \ldots, m$ there exists $L_i > 0$ such that
	\begin{equation*}
		- L_i \| z - y\|^2 \leq h_i(z) - h_i(y) - \left \langle \nabla h_i^T(z), z - y \right \rangle \, .
	\end{equation*}
	In particular, for $z \in \calX_i$ (i.e., $h_i(z) \leq 0$) and $y \in \partial \calX_i$ (i.e., $h_i(y) = 0$) in a neighborhood of $x$ we have
	\begin{equation}\label{eq:pre_prox}
		\left\langle \nabla h_i^T(y), z-y \right\rangle \leq h_i(z)+ L_i \| z - y\|^2 \leq L_i \| z - y \|^2 \, .
	\end{equation}

	For the set $\calX = \bigcap_{i=1}^m \calX_i$ recall from \cref{ex:clarke_reg_normal_cone} that for $x \in \calX$ we have
	\begin{align*}
		N_x \calX = \left\lbrace \eta \, \middle| \, \eta = \sum\nolimits_{i \in I(x)} \alpha_i {\nabla h^T_i(x)}, \, \alpha_i \geq 0 \right\rbrace \, .
	\end{align*}
	Consider $z \in \calX$ and $y \in \partial \calX$ in a small enough neighborhood of $x$. Note that $y \in \partial \calX$ implies that $y \in \partial \calX_i$ for all $i \in I(y)$. Using~\eqref{eq:pre_prox}, for all $\eta \in N_y\calX$ with $\eta = \sum_{i \in I(y)} \alpha_i {\nabla h^T_i(y)}/ \| \nabla h_i(y) \|$ we have
	\begin{equation*}
		\left\langle \eta, z- y \right\rangle
		= \left\langle
		\sum\nolimits_{i \in I(y)} \alpha_i {\nabla h_i(y)}^T, z - y \right\rangle
		\leq
		\left(
		\sum\nolimits_{i \in I(y)} \alpha_i  L_i \, ,
		\right) \| z - y\|^2 \,
	\end{equation*}
	and therefore $
		\langle \eta, z - y \rangle
		\leq L(y) \| \eta \| \| z - y\|^2$, where
	\begin{align*}
		L(y) := \tfrac{\sum\nolimits_{i \in I(y)} \alpha_i L_i}{\| \eta \|} = \tfrac{\sum\nolimits_{i \in I(y)} \alpha_i L_i}{\left \| \sum\nolimits_{i \in I(y)} \alpha_i \nabla h_i(y) \right\|}
		\leq \max_{i \in I(y)} \tfrac{\alpha_i \nabla L_i}{\alpha_i \| \nabla h_i(y) \|} \leq \max_{i=1, \ldots m} \tfrac{L_i}{\| \nabla h_i(y) \|} \, .
	\end{align*}
	The first inequality can be shown by taking the square and proceeding by induction. Since the final bound is with respect to all $h_i$, it is continuous in $y$ in a neighbhorhood of $x$. Consequently, we can choose $\bar{L}$ such that $\bar{L} \geq L(y)$ for all $y \in \calX$ in a neighborhood of $x$, and therefore $\langle \eta, z - y \rangle
		\leq \bar{L} \| \eta \| \| z - y\|^2 $ for $z \in \calX$ in a neighborhood of $y$. This proves $\bar{L}$-prox-regularity at $x$ and prox-regularity follows accordingly.
\end{example}

\subsection{Uniqueness of solutions to projected dynamical systems}
Before formulating our main uniqueness result, we present an example that illustrates the impact of prox-regularity on the uniqueness of solutions.

\begin{example}[prox-regularity and uniqueness of solutions]\label{ex:prox_uniq}
	We consider the set $\calX_\alpha := \left\lbrace (x_1, x_2) \, \middle | \, |x_2| \geq {\max \{0, x_1\}}^\alpha\right\rbrace$ for $0 < \alpha < 1$, as in \cref{ex:prox_reg_basic}.
	We study how the value of $\alpha$ affects the uniqueness of solutions of the projected dynamical system defined by the uniform ``horizontal'' vector field $f(x) = (1, 0)$ for all $x \in \calX$ and the initial condition $x(0) = 0$ as illustrated in \cref{fig:prox_uniq}.

	\begin{figure}[bt]
		\centering
		\begin{subfigure}[t]{0.2\textwidth}
			\centering
			\includegraphics[width=\textwidth]{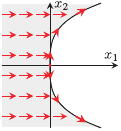}
			\subcaption{$\alpha = 0.3$}
		\end{subfigure} \hspace{.05\textwidth}
		\begin{subfigure}[t]{0.2\textwidth}
			\centering
			\includegraphics[width=\textwidth]{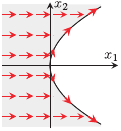}
			\subcaption{$\alpha = 0.5$}
		\end{subfigure} \hspace{.05\textwidth}
		\begin{subfigure}[t]{0.2\textwidth}
			\centering
			\includegraphics[width=\textwidth]{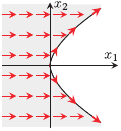}
			\subcaption{$\alpha = 0.6$}
		\end{subfigure}
		\caption{Projected vector field on $\calX_\alpha$ for different values of $\alpha$ in \cref{ex:prox_uniq}. The origin is a strong equilibrium in (a) and (b), and it is a weak equilibrium in (c).}\label{fig:prox_uniq}
	\end{figure}

	Since $\calX_\alpha$ is Clarke regular and closed, since the vector field is uniform, and since we use the Euclidean metric, the existence of Krasovskii solutions and the equivalence of Carath\'eodory solutions is guaranteed for $t\rightarrow \infty$ by \cref{cor:max_sol} and \cref{thm:main_equiv}, respectively.
	The prox-regularity of $\calX_\alpha$ at the origin is however only guaranteed for $0 < \alpha \leq \frac{1}{2}$ (\cref{ex:prox_reg_basic}).

	A formal analysis reveals that for $0 < \alpha \leq \frac{1}{2}$ the origin is a strong equilibrium, i.e., the constant solution $x(t) = 0$ is the unique solution to the projected dynamical system. For $\frac{1}{2} < \alpha < 1$, however, the origin is only a weak equilibrium point. Namely, a solution may remain at $0$ for an arbitrary amount of time before leaving $0$ on either upper or lower halfplane, and thus uniqueness is not guaranteed.
\end{example}

\begin{remark}\label{rem:unique_proj}
	Whether $\tproj{}{}{f}(x_0)$ is a singleton or not is generally unrelated to the uniqueness of solutions starting from $x_0$. For instance, in \cref{ex:prox_uniq}, if $\alpha > 0$ multiple solutions exists even though $\tproj{}{}{f}(x)$ is a singleton at $x = 0$. Conversely, \cref{ex:marble_run_krasovskii} shows that even if $\tproj{}{}{f}(x_0)$ is not unique, the (Krasovskii) solution starting from~$x_0$ is unique.
\end{remark}

\ifARXIV
	For the proof of uniqueness under prox-regularity, we require the following lemma.
\else
	For the proof of uniqueness under prox-regularity, we require the following technical lemma whose proof is purely technical and can be found in~\cite{hauswirthProjectedDynamicalSystems2018a}.
\fi

\begin{lemma}\label{lem:hypomonotone} Let $\calX$ be $L$-prox-regular at $x$ with respect to a $C^{0,1}$ metric $g$. Then, there exist $\bar{L} > 0$ such that for all $y \in \calX$ in a neighborhood of $x$ and all $\eta \in N_y^{g,L}$ with $\| \eta \|_{g(y)} = 1$ we have $\left\langle \eta, x - y \right\rangle_{g(x)} \leq \bar{L} \| y - x \|^2_{g(x)}$.
\end{lemma}

\ifARXIV
	\begin{proof} We know that $ \left\langle \eta, y - x \right\rangle_{g(y)} \leq L \| y - x \|_{g(y)}^2 $ for $y$ close enough to $x$ because $\eta$ is a $L$-proximal normal vector at $y$ with respect to $g$. Furthermore, by the equivalence of norms there exists $L' > 0$ sucht that $ \left\langle \eta, y - x \right\rangle_{g(y)} \leq L' \| y - x \|_{g(x)}^2$.

		Next, we show that $| \left\langle \eta, x - y \right\rangle_{g(y)} - \left\langle \eta, x -y \right\rangle_{g(x)} | \leq M \| y - x \|^2_{g(x)}$ for some $M > 0$. Since $L^n_2$ is a vector space, we may write
		\begin{equation*}
			\left\langle \eta, x-y \right\rangle_{g(y)} - \left\langle \eta, x-y \right\rangle_{g(x)} = \left\langle \eta, x-y \right\rangle_{g(y) - g(x)}
		\end{equation*}
		which is a slight abuse of notation since $\left\langle \cdot, \cdot \right\rangle_{g(y) - g(x)}$ is not necessarily positive definite and therefore not a metric. Nevertheless, any map of the form $(u,v,g) \mapsto \left\langle u, w \right\rangle_{g}$ where $g \in L^n_2$ is linear in $u,v$ and in $g$ (e.g., $(u,v,g) \mapsto \left\langle u, w \right\rangle_{\lambda g} = \lambda \left\langle u, w \right\rangle_{g}$ for any $\lambda \in \bbR$). Therefore, there exist $M', M > 0$ such that
		\begin{equation*}
			\left| \left\langle \eta, x-y \right\rangle_{g(y) - g(x)} \right | \leq M' \| g(y) - g(x) \|_{L^n_2} \| x - y \|_{g(x)} \leq M \| x - y \|^2_{g(x)} \, ,
		\end{equation*}
		where $\| \cdot \|_{L^n_2}$ denotes any norm on the vector space $L^n_2$, and the second inequality follows directly from the Lipschitz continuity of $g$. Hence, we can conclude that that
		\begin{align*}
			\left\langle \eta, x -y \right\rangle_{g(x)}
			\leq \left\langle \eta, x -y \right\rangle_{g(y)}
			+ | \left\langle \eta, x - y \right\rangle_{g(y) - g(x)} |
			\leq (L' + M) \| y - x \|^2_{g(x)} \, .
		\end{align*}
	\end{proof}
\fi

Next, we can show the following Lipschitz-type property of projected vector fields.

\begin{proposition}\label{prop:lipschitz_flat} Let $f$ be a $C^{0,1}$ field on $\calX$. If $g$ is a $C^{0,1}$ metric and $\calX$ is prox-regular, then for every $x \in \calX$ there exists $L > 0$ such that for all $y \in \calX$ in a neighborhood of $x$ we have
	\begin{equation*}
		\left\langle \tproj{\calX}{g}{f}(y) - \tproj{\calX}{g}{f}(x), y - x \right\rangle_{g(x)} \leq L \| y - x \|^2_{g(x)} \, .
	\end{equation*}
\end{proposition}

\begin{proof} As a consequence of \cref{lem:moreau_gen}, we can write
	\begin{multline}\label{eq:onesided_lip}
		\left\langle \tproj{\calX}{g}{f}(y) - \tproj{\calX}{g}{f}(x), y - x \right\rangle_{g(x)} \\
		= \left\langle f(y) - f(x) , y - x \right\rangle_{g(x)}
		+ \left\langle \eta_y, x - y \right\rangle_{g(x)}
		+ \left\langle \eta_x, y - x \right\rangle_{g(x)} \, .
	\end{multline}
	where $\eta_y \in N^g_y \calX = \bar{N}_y^{g, L} \calX$ and $\eta_x \in N^g_x \calX = \bar{N}_x^{g, L}$ for some $L> 0$.

	For the first term, we get
	$\left\langle f(y) - f(x) , y - x \right\rangle_{g(x)} \leq \| f(y) - f(x) \|_{g(x)} \| y - x\|_{g(x)}$.
	by applying Cauchy-Schwarz.
	Since $f$ is Lipschitz and using the equivalence of norms there exists $L_a > 0$ such that
	$\| f(y) - f(x) \|_{g(x)} \leq L_a \| y - x\|_{g(x)}$
	for all $y \in \calX$ in a neighborhood of $x$. Thus, we have
	$\left\langle f(y) - f(x) , y - x \right\rangle_{g(x)} \leq L_a \| y - x\|^2_{g(x)}$.

	For the second and third term in~\eqref{eq:onesided_lip} we have
	\begin{align*}
		\left\langle \eta_y, x - y \right\rangle_{g(x)} & \leq L' \| y - x \|_{g(x)}^2 \| \eta_y \|_{g(y)} \\
		\left\langle \eta_x, y - x \right\rangle_{g(x)} & \leq L \| y - x \|_{g(x)}^2 \| \eta_x \|_{g(x)}
	\end{align*}
	by \cref{lem:hypomonotone} and the definition of a $L$-proximal normal vector, respectively.

	By \cref{lem:moreau_gen} we know that $\| \eta_y \|_{g(y)} \leq \| f(y) \|_{g(y)}$ and $\| \eta_x \|_{g(x)} \leq \| f(x) \|_{g(x)}$. Since $g$ and $f$ are continuous we can choose $M > 0$ such that $\| f(z) \|_{g(z)} \leq M$ for all $z \in \calX$ in a neighborhood of $x$. Therefore,~\eqref{eq:onesided_lip} can be bounded by
	\begin{align*}
		\left\langle \tproj{\calX}{g}{f}(y) - \tproj{\calX}{g}{f}(x), y - x \right\rangle_{g(x)}
		\leq (L_a + L' M + L M)\| y - x \|^2_{g(x)}
	\end{align*}
	which completes the proof.
\end{proof}

Hence, we can state our main result on the uniqueness of solutions which complements results in~\cite{cornetExistenceSlowSolutions1983} by considering a variable (but non-differentiable) metric and using our general definition of prox-regularity. In this context, uniqueness is understood in the sense that any two solutions are equal on the interval on which they are both defined.

\begin{theorem}[uniqueness of solutions]\label{thm:main_uniq} Let $f$ be a $C^{0,1}$ vector field on $\calX$. If $g$ is a $C^{0,1}$ metric and $\calX$ is prox-regular, then for every $x_0 \in \calX$ there exists $T> 0$ such that the initial value problem
	$\dot x \in \tproj{\calX}{g}{f}(x)$ with $x(0) = x_0$
	has a unique Carath\'eodory solution $x:[0, T) \rightarrow \calX$ (which is also the unique Krasovskii solution).
\end{theorem}

\begin{proof}[Proof of \cref{thm:main_uniq}] The proof follows standard contraction ideas~\cite{filippovDifferentialEquationsDiscontinuous1988}. Let $x(t)$ and $y(t)$ be two solutions solving the same initial value problem $\dot x \in \tproj{\calX}{g}{f}(x)$ with $x(0) = x_0 \in \calX$, both defined on a non-empty interval $[0, T)$.

	Using \cref{prop:lipschitz_flat}, there exists $M > 0$ and a neighborhood $V$ of $x_0$ such that
	\begin{equation}\label{eq:lipschitz_uniq}
		\begin{split}
			\tfrac{d}{dt} \left( \tfrac{1}{2}\| y(t) - x(t) \|^2_{g(x_0)} \right) & = \left\langle \tproj{\calX}{g}{ f(y}(t)) - \tproj{\calX}{g}{ f(x}(t)), y(t) - x(t) \right\rangle_{g(x_0)} \\ & \leq M || y(t) - x(t) ||^2_{g(x_0)}
		\end{split}
	\end{equation}
	for all $t$ in some non-empty subinterval $[0, T') \subset [0, T)$ for which $x(t)$ and $y(t)$ remain in $V$.
	Next, consider the non-negative, absolutely continuous function $q: [0, T') \rightarrow \bbR$ defined as $q(t) := \frac{1}{2} \| y(t) - x(t) \|^2_{g(x_0)} e^{-2M t}$. Note that $q(0) = 0$. Furthermore, using~\eqref{eq:lipschitz_uniq} and applying the product rule we have
	\begin{align*}
		\tfrac{d}{dt} q(t) & = ( \left\langle \tproj{\calX}{g}{ f(y}(t)) - \tproj{\calX}{g}{ f(x}(t)), y(t) - x(t) \right\rangle_{g(x_0)} - M || y(t) - x(t) ||^2_{g(x_0)} ) e^{-2M t}
	\end{align*}
	and since $y(0) = x(0)$ it follows that $\frac{d}{dt} q(t) \leq 0$ for $t \geq 0$. However, since $q$ is non-negative and absolutely continuous, we conclude that $x(t) = y(t)$ for all $t \in [0, T')$ thus finishing the proof of uniqueness.
\end{proof}

Combining all the insights so far, we arrive at the following ready-to-use result:

\begin{example}[Existence and uniqueness on constraint-defined sets] As in \cref{ex:clarke_reg_prox_cone} consider a set $\calX := \{ x \in \bbR^n \, | \, h(x) \leq 0 \}$ where $h: \bbR^n \rightarrow \bbR^m$ is of class $C^{1,1}$ and has full rank for all $x \in \bbR^n$. Further, consider a globally Lipschitz continuous vector field $f: \bbR^n \rightarrow \bbR$.
	Then, for every $x_0 \in \calX$ there exists a unique and complete Carath\'eodory solution $x: [0, \infty) \rightarrow \calX$ to the initial value problem $\dot x = \tproj{\calX}{g}{f}(x)$ with $x(0) = x_0$ where $g$ is any weakly bounded $C^{0,1}$ metric on $\calX$.
\end{example}

\section{Existence and Uniqueness on low-regularity Riemannian Manifolds}\label{sec:mfd}

The major appeal of \cref{thm:main_exist,thm:main_equiv,thm:main_uniq} is their geometric nature. Namely, as we will show next, their assumptions are preserved by sufficiently regular coordinate transformations which allows us to give a coordinate-free definition of projected dynamical system on manifolds with minimal degree of differentiability.

Recall that for open sets $V, W \subset \bbR^n$ a map $\Phi: V \rightarrow W$ is a \emph{$C^k$ diffeomorphism} if it is a $C^k$ bijection with a $C^k$ inverse where, for our purposes, $C^k$ stands for either $C^1$ or $C^{1,1}$.
We employ the usual definition of a \emph{$C^k$ manifold} as locally Euclidean, second countable Hausdorff space endowed with a $C^k$ differentiable structure. In particular, for a point $p$ on a $n$-dimensional manifold $\calM$ there exists a chart $(U, \phi)$ where $U \subset \calM$ is open and $\phi: U \rightarrow \bbR^n$ is a homeomorphism onto its image. For any two charts $(U, \phi), (V, \psi)$ for which $U \cap V \neq \emptyset$, the map $\phi \circ \psi^{-1}: \psi(U \cap V) \rightarrow \phi(U \cap V)$ is a $C^k$ diffeomorphism.
A \emph{$C^k$ (Riemannian) metric $g$} is a map that assigns to every point $p \in \calM$ an inner product on the \emph{tangent space}\footnote{Note that the definition (and hence the notation) of the tangent space $T_x \calM$ of a manifold $\calM$ is consistent with the definition of the tangent cone $T_x \calX$ of an arbitrary set $\calX$~\cite[Ex.~6.8]{rockafellarVariationalAnalysis2009}.}
$T_p \calM$ such that in local coordinates $(U, \phi)$ the metric $g(\phi^{-1}(x))$ is a $C^k$ metric for $x \in \phi(U)$ according to \cref{def:metric}.
A vector field defined on $\calM$ is \emph{locally bounded at $x$} if it is locally bounded in any local coordinate domain for $x$. Similarly, a metric is \emph{locally weakly bounded at $x$} if its locally weakly bounded in local coordinates.
Given a $C^k$ manifold $\mathcal{M}$ with $k\geq 1$, a curve $\gamma:[0, T) \rightarrow \mathcal{M}$ is \emph{absolutely continuous} if it is absolutely continuous in any chart domain where it is defined.\footnote{Note that local (weak) boundedness of a vector field or metric are properties that are preserved by $C^1$ diffeomorphisms. Similarly, absolute continuity is preserved by $C^1$ maps~\cite[Ex.~6.44]{roydenRealAnalysis1988}. Hence, it is sufficient if these properties hold in any local coordinate domain.}

The next lemma shows that a $C^1$ diffeomorphism maps (Clarke) tangent cones to (Clarke) tangent cones. Hence, Clarke regularity is preserved by $C^1$ diffeomorpisms.
\unless\ifARXIV
	The proof simple but technical and can be found in~\cite{hauswirthProjectedDynamicalSystems2018a}.
\fi

\begin{lemma}\label{lem:clarke_reg_trans} Let $V, W \subset \mathbb{R}^n$ be open and consider a $C^1$ diffeomorphism $\Phi: V \rightarrow W$. Given $\calX \subset \bbR^n$ and $\tilde{\calX} := \calX \cap V$, for every $x \in \tilde{\calX}$ it holds that
	\begin{align}\label{eq:tgt_equiv}
		T_{\Phi(x)}\Phi(\tilde{\calX})   & = D_x\Phi (T_x \tilde{\calX})        \\\label{eq:ctgt_equiv}
		T^C_{\Phi(x)}\Phi(\tilde{\calX}) & = D_x\Phi (T^C_x \tilde{\calX}) \, .
	\end{align}
	Hence, $\Phi(\tilde{\calX})$ is Clarke regular at $\Phi(x)$ if and only if $\tilde{\calX}$ is Clarke regular at $x \in \tilde{\calX}$.
\end{lemma}

\ifARXIV
	\begin{proof}
		We only need to show that
		$T_{\Phi(x)}\Phi(\tilde{\calX}) \subset D_x\Phi (T_x \tilde{\calX})$.
		Since $\Phi$ is a $C^1$ diffeomorphism the other direction follows by applying the same arguments to $\Phi^{-1}$.

		Let $v \in T_x \tilde{\calX}$. Then, by definition there exist $x_k \rightarrow x$ with $x_k \in \tilde{\calX}$ and $\delta_k \rightarrow 0^+$ such that $(x_k - x)/ \delta_k \rightarrow v$. Furthermore, $\| x_k - x \| / \delta_k$ converges to $\| v \|$. According to the definition of the derivative of $\Phi$, for the same sequence $\{ x_k \}$ we have
		$\underset{k \rightarrow \infty}{\lim} \| \Phi(x_k) - \Phi(x) - D_x \Phi(x_k - x) \| / \| x_k - x \| = 0 $.
		Since the limit of the element-wise product of convergent sequences equals the product of its limits we can write
		\begin{equation*}
			\underset{k \rightarrow \infty}{\lim} \tfrac{\left \| \Phi (x_k) - \Phi(x) - D_x \Phi(x_k - x) \right \|}{\| x_k - x \|} \tfrac{\|x_k - x \|}{\delta_k} = 0
		\end{equation*}
		which, using the fact that $D_x \Phi$ is linear, simplifies to
		\begin{equation*}
			\underset{k \rightarrow \infty}{\lim} \left \| \tfrac{\Phi(x_k) - \Phi(x)}{\delta_k} - D_x \Phi\left(\tfrac{x_k - x}{\delta_k} \right) \right \| = 0 \, .
		\end{equation*}
		This implies that $(\Phi(x_k) - \Phi(x))/\delta_k \rightarrow D_x \Phi (v)$, and hence $D_x \Phi(v)$ is a tangent vector of $\Phi(\tilde{\calX})$ at $\Phi(x)$. This proves~\eqref{eq:tgt_equiv}.

		To show~\eqref{eq:ctgt_equiv} we use~\eqref{eq:tgt_equiv} together with the definition of the Clarke tangent cone as the inner limit of the surrounding tangent cones (\cref{def:clarke_tgt}). We can write
		\begin{equation*}
			T_{\Phi(x)}^C \Phi(\tilde{\calX})  =
			\underset{\hat{y} \rightarrow \Phi(x)}{\lim \inf} \,
			T_{\hat{y}} \Phi(\tilde{\calX})
			= \underset{y \rightarrow x}{\lim \inf} \,
			D_{y} \Phi \left( T_{y} \tilde{\calX} \right) \, .
		\end{equation*}
		Since $D_x \Phi$ is continuous in $x$, we have
		$\underset{y \rightarrow x}{\lim \inf} \,
			D_{y} \Phi ( T_{y} \tilde{\calX} ) = \underset{y \rightarrow x}{\lim \inf} \,
			D_{x} \Phi ( T_{y} \tilde{\calX} )$.
		Further, \cref{lem:continuity_set} implies that
		$\underset{y \rightarrow x}{\lim \inf} \,
			D_{x} \Phi ( T_{y} \tilde{\calX} ) \supset
			D_{x} \Phi ( \underset{y \rightarrow x}{\lim \inf} \, T_{y} \tilde{\calX})
			= D_x\Phi (T^C_x \tilde{\calX} ) $
		and therefore we have $ T_{\Phi(x)}^C \Phi(\tilde{\calX}) \supset D_x\Phi (T^C_x \tilde{\calX})$. Again, since $\Phi$ is a diffeomorphism, the opposite inclusion holds by applying the same argument to $\Phi^{-1}$. This shows~\eqref{eq:ctgt_equiv} and completes the proof.
	\end{proof}
\fi

Hence, the notions of (Clarke) tangent cone and Clarke regularity are independent of the coordinate representation on a $C^1$ manifold.

\begin{definition} Let $\mathcal{M}$ be a $C^1$ manifold with a metric $g$ and consider a subset $\calX \subset \mathcal{M}$. The (Clarke) tangent cone $T_x \calX$ $(T^C_x \calX)$ is a subset of $T_x \calM$ such that $D_x \phi(T_x \calX)$ $(D_x \phi(T^C_x \calX))$ is the (Clarke) tangent cone of $\phi(\calX \cap U)$ for any coordinate chart $(U, \phi)$ defined at $x$. The set $\calX$ is Clarke regular at $x \in \calX$ if it is Clarke regular in any local coordinate domain defined at $x$.
\end{definition}

The next key result establishes that solutions of projected dynamical systems remain solutions of projected dynamical systems under $C^{1}$ coordinate transformations.

\begin{proposition}\label{prop:inv_pds} Let $V, W \subset \bbR^n$ be open and consider a $C^{1}$ diffeomorphism $\Phi: V \rightarrow W$. Let $\calX \subset \bbR^n$ be locally compact and $\tilde{\calX} := \calX \cap V$. Further, let $g$ be a locally weakly bounded metric on $W$ and let $\Phi^* g$ denote the \emph{pull-back metric along $\Phi$}, i.e.,
	\begin{equation}\label{eq:pullback_metric}
		\left\langle v, w \right\rangle_{\Phi^* g(x)} := \left\langle D_{x} \Phi (v), D_{x} \Phi (w) \right\rangle_{g(\Phi(x))}
	\end{equation}
	for all $x \in V$ and $v,w \in T_x \bbR^n$. Further, let $f: \tilde{\calX} \rightarrow \mathbb{R}^n$ be a locally bounded vector field. If $x:[0, T) \rightarrow \tilde{\calX}$ for some $T> 0$ is a Krasovskii (respectively, Carath\'eodory) solution to the initial value problem
	\begin{equation}\label{eq:non_trans_prob}
		\dot x \in \tproj{\tilde{\calX}}{\Phi^*g}{f}(x) \, , \quad x(0) = x_0 \, ,
	\end{equation}
	then $\Phi \circ x: [0, T) \rightarrow \Phi(\tilde{\calX})$ is a Krasovskii (respectively, Carath\'eodory) solution to
	\begin{equation}\label{eq:transf_prob}
		\dot{y} \in \tproj{{g}}{{\Phi(\tilde{\calX})}}{\hat{f}}(y) \, , \quad y(0) = y_0 \,,
	\end{equation}
	where $y_0 := \Phi(x_0)$ and $\hat{f}(y) := D_{\Phi^{-1}(y)} \Phi (f ( \Phi^{-1}(y)))$ is the \emph{pushforward vector field of $f$ along $\Phi^{-1}$}.
\end{proposition}

\begin{proof} First, note that since $x$ is absolutely continuous and $\Phi$ is differentiable, $\Phi\circ x$ is absolutely continuous~\cite[Ex.~6.44]{roydenRealAnalysis1988}. Second, it holds that $y(t) \in \Phi(\tilde{\calX})$ for all $t \in [0, T)$. Third, using~\eqref{eq:tgt_equiv} we can write for every $x \in \tilde{\calX}$ and $y := \Phi(x)$ that
	\begin{align*}
		\tproj{{g}}{{\Phi(\tilde{\calX})}}{\hat{f}}(y)
		 & = \underset{w \in T_{y} \Phi(\tilde{\calX})}{\arg \min}
		\left \|w - D_{x} \Phi(f(x)) \right \|_{g}
		= \underset{w \in D_{x} \Phi \left(T_{x} \tilde\calX\right)}{\arg \min}
		\left \|w - D_{x} \Phi(f(x)) \right \|_{g}                          \\
		 & = D_{x}\Phi \left(\underset{v \in T_{x}\tilde{\calX}}{\arg \min}
		\left \| D_{x}\Phi(v) - D_{x} \Phi(f(x)) \right \|_{g} \right) \, ,
	\end{align*}
	where for the last equality we introduce the transformation $w := D_x \Phi(v)$ for $v \in T_{x}\tilde\calX$. Hence, using the definition of the pullback metric~\eqref{eq:pullback_metric} we continue with
	\begin{align*}
		\tproj{{g}}{{\Phi(\tilde{\calX})}}{\hat{f}}(y)
		= D_{x}\Phi \left( \underset{ v \in T_{x}{\tilde{\calX}}}{\arg \min}
		\left \|v - f(x) \right \|_{\Phi^* g} \right) = D_{x} \Phi \left(\tproj{\tilde{\calX}}{\Phi^*g}{f}(x)\right) \, .
	\end{align*}
	Thus, if $x(\cdot)$ is a Carath\'eodory solution of~\eqref{eq:non_trans_prob} and hence
	$\dot x(t) \in \tproj{\tilde{\calX}}{\Phi^*g}{f}(x(t))$
	holds almost everywhere, then $\Phi \circ x(\cdot)$ satisfies
	\begin{equation*}
		\frac{d}{dt}\left(\Phi \circ x \right) \in D_{x} \Phi \left(\tproj{\tilde{\calX}}{\Phi^*g}{f}(x)\right) = \tproj{\Phi(\tilde{\calX})}{g}{\hat{f}}(\Phi \circ x(t))
	\end{equation*}
	almost everywhere and hence $\Phi \circ x(\cdot)$ is a Carath\'eodory solution to~\eqref{eq:transf_prob}.

	It remains to prove the statement is also true for Krasovskii solutions. For this, we need to show that
	$\Kras{ \tproj{{g}}{{\Phi(\tilde{\calX})}}{\hat{f}}}(y)
		\supset D_{x} \Phi (\Kras{\tproj{\tilde{\calX}}{\Phi^*g}{f } }(y) )$.
	Expanding the definition of the Krasovskii regularization we get
	\begin{align*}
		\Kras{ \tproj{{g}}{{\Phi(\tilde{\calX})}}{\hat{f}}}(y)
		 & =
		\cocl \, \underset{\tilde{y} \rightarrow y}{\lim \sup} \, \tproj{\Phi(\tilde{\calX})}{g}{\hat{f}}(\tilde{y}) \\
		 & =
		\cocl \, \underset{\tilde{x} \rightarrow x}{\lim \sup} \,
		D_{\tilde{x}} \Phi \left( \tproj{\tilde{\calX}}{\Phi^*g}{f }(\tilde{x}) \right)                              \\
		 & =
		\cocl \, \underset{\tilde{x} \rightarrow x}{\lim \sup} \,
		D_{x} \Phi \left( \tproj{\tilde{\calX}}{\Phi^*g}{f }(x_k) \right) \, ,
	\end{align*}
	where the last equation is due to the fact that $D_x \Phi$ is continuous in $x$. Next, with \cref{lem:continuity_set} we can write
	\begin{align*}
		\Kras{ \tproj{{g}}{{\Phi(\tilde{\calX})}}{\hat{f}}}(y)
		 & \supset
		\cocl \, D_{x} \Phi \left(\underset{\tilde{x} \rightarrow x}{\lim \sup} \,
		\tproj{\tilde{\calX}}{\Phi^*g}{f }(x_k) \right)
		= D_{x} \Phi \left(\Kras{\tproj{\tilde{\calX}}{\Phi^*g}{f } }(x) \right)
	\end{align*}
	where the equation follows from the fact that $D_x \Phi$ is a linear map and hence commutes with taking the convex closure.

	To conclude we can proceed similar to the case of Carath\'eodory solutions. Let $x(\cdot)$ be a Krasovskii solution to~\eqref{eq:non_trans_prob} and $y(\cdot) := \Phi \circ x(\cdot)$. Then, $\dot y(t) = \frac{d}{dt}(\Phi \circ x)(t) = D_{x(t)} \Phi( \dot x(t))$ for almost all $t \in [0, T)$ and we have that
	\begin{equation*}
		\dot y(t) \in
		D_{x(t)} \Phi \left(\Kras{\tproj{\tilde{\calX}}{\Phi^*g}{f }( x}(t) ) \right) \subset
		\Kras{ \tproj{\Phi(\tilde{\calX})}{g}{\hat{f}}}(y(t))
	\end{equation*}
	for almost all $t \in [0, T)$, and thus $y$ is a Krasovskii solution of~\eqref{eq:transf_prob}.
\end{proof}

Hence, \cref{thm:main_exist,thm:main_equiv} combined with \cref{prop:inv_pds} give rise to our main result on the existence of Krasovskii (Careth\'eodory) solutions to on manifolds.

\begin{theorem}[existence on manifolds]\label{thm:main_mfd}
	Let $\mathcal{M}$ be $C^1$ manifold, $g$ a locally weakly bounded Riemannian metric, $\calX \subset \mathcal{M}$ locally compact, and $f$ a locally bounded vector field on $\calX$. Then for every $x_0 \in \calX$ there exists a Krasovskii solution $x: [0, T) \rightarrow \calX$ for some $T> 0$ that solves
	$\dot x(t) \in \tproj{\calX}{g}{f}(x(t))$ with $x(0) = x_0$.
	Furthermore, if $\calX$ is Clarke regular, and if $f$ and $g$ are continuous, then every Krasovskii solution is a Carath\'eodory solution and vice versa.
\end{theorem}

Similarly, \cref{prop:inv_pds} directly implies that other results such as \cref{cor:max_sol} extend to $C^1$ manifolds. For instance, if $\calM$ is compact and $f$ and $g$ are continuous, every initial condition admits a complete trajectory. However, to extend our uniqueness results, we require stronger conditions.

\begin{proposition}\label{prop:c11_prox} Let $V,W \subset \bbR^n$ be open and $\Phi: V \rightarrow W$ a $C^{1,1}$ diffeo\-morphism. Let $\calX \subset \bbR^n$ be locally compact and consider $\tilde{\calX} := \calX \cap V$. If $\tilde{\calX}$ is prox-regular then $\Phi(\tilde{\calX})$ is prox-regular.
\end{proposition}

\begin{proof} By \cref{prop:prox_invariance} it suffices to show prox-regularity with respect to a single metric on $V$ and $W$ respectively. Hence, let $W$ be endowed with the Euclidean metric, and let $e^*$ denote its pullback metric on $V$ along $\Phi$, i.e., $\left\langle v, w \right\rangle_{e^*(x)} := \left\langle D_x\Phi(v), D_x\Phi(w) \right\rangle$.
	Similarly to \cref{lem:clarke_reg_trans}, we show that (proximal) normal cones are preserved by $C^1$ coordinate transformations, i.e.,
	\begin{align}
		\eta \in N^{e^*}_x \tilde{\calX}          & \quad \Longleftrightarrow \quad D_x\Phi(\eta) \in N_{\Phi(x)} \Phi(\tilde{\calX}) \qquad \forall x \in \tilde{\calX}\label{eq:norm_impl}                 \\
		\eta \in \bar{N}^{e^*, L}_y \tilde{\calX} & \quad \Longleftrightarrow \quad D_y\Phi(\eta) \in \bar{N}^{L'}_{\Phi(y)} \Phi(\tilde{\calX}) \qquad \forall y \in \mathcal{N}_x\label{eq:norm_impl_prox}
	\end{align}
	for some $L', L > 0$ where $\mathcal{N}_x \subset \tilde{\calX}$ is a neighborhood of $x$. Since $\Phi$ is a diffeomorphism it suffices to show one direction only.

	Hence, consider $\eta \in N^{e^*}_x \tilde{\calX}$. By \cref{def:norm_cone} and using~\eqref{eq:tgt_equiv} we have
	\begin{align*}
		\eta \in N^{e^*}_x \tilde{\calX} \quad & \Leftrightarrow \quad \left\langle \eta, w \right\rangle_{e^*(x)} = \left\langle D_x\Phi(\eta), D_x\Phi(w) \right\rangle \leq 0 \quad \forall w \in T_x \tilde{\calX} \\
																					 & \Leftrightarrow \quad \left\langle D_x\Phi(\eta), w \right\rangle \leq 0 \quad \forall w \in D_{x} \Phi (T_x \tilde{\calX}) = T_{\Phi(x)} \Phi(\tilde{\calX}) \, .
	\end{align*}
	We conclude that $D_x\Phi(\eta) \in N_{\Phi(x)} \Phi(\tilde{\calX})$ and~\eqref{eq:norm_impl} holds.

	For~\eqref{eq:norm_impl_prox} we consider $y \in \tilde{\calX}$ in a neighborhood of $x$ and $\eta \in \bar{N}_y^{e^*, L} \tilde{\calX}$ such that
	\begin{align*}
		\left\langle \eta, z - y \right\rangle_{e^*(y)}
		=
		\left\langle D_y\Phi(\eta), D_y\Phi (z - y) \right\rangle
		\leq L \| z - y \|^2_{e^* g(y)}
	\end{align*}
	holds for all $z \in \tilde{\calX}$ in a neighborhood of $y$. However, we need to show that for some $L' > 0$ we have
	\begin{align}\label{eq:prox_goal}
		\left\langle D_y\Phi(\eta), \Phi(z) - \Phi(y) \right\rangle
		\leq L' \| \Phi(z) - \Phi(y) \|^2 \, .
	\end{align}

	Hence, we define the $C^{1,1}$ function $\psi(z) := \left\langle D_y\Phi(\eta), \Phi(z) \right\rangle$ and note that by linearity we have $D_z \psi (v) := \left\langle D_y\Phi(\eta), D_z\Phi(v) \right\rangle$. This enables us to apply the Desent \cref{lem:c11_lipschitz} and state that for some $M > 0$ it holds that
	\begin{align*}
		| \psi(z) - \psi(y) - D_y \psi(z - y) | =
		\underbrace{| \left\langle D_y\Phi(\eta), \Phi(z) - \Phi(y) - D_y \Phi(z - y) \right\rangle |}_{ =: \gamma(z)}
		\leq M \| z - y \|^2 \, .
	\end{align*}
	This bound can be used to establish
	\begin{align*}
		\left\langle D_y\Phi(\eta), \Phi(z) - \Phi(y) \right\rangle \leq
		\left\langle D_y\Phi(\eta), D_y\Phi (z - y) \right\rangle
		+ \gamma(z) \leq (L + M) \| z - y \|^2 \, .
	\end{align*}

	Finally note that $\| z - y \|^2 \leq L' \| \Phi(z) - \Phi(y) \|^2$ for some $L'$ since $\Phi^{-1}$ is Lipschitz continuous. Hence,~\eqref{eq:prox_goal} and therefore~\eqref{eq:norm_impl_prox} holds for $L' = L'' ( L + M)$.
\end{proof}

Apart from \cref{prop:c11_prox}, we note that Lipschitz continuity of a metric and of vector fields is preserved under $C^{1,1}$ coordinate transformations. This allows us to generalize \cref{thm:main_uniq} to the following uniqueness result on manifolds.

\begin{theorem}[uniqueness on manifolds]\label{thm:main_mfd_uniq}
	Let $\mathcal{M}$ be $C^{1,1}$ manifold, $g$ a $C^{0,1}$ Riemannian metric, $\calX \subset \mathcal{M}$ is prox-regular, and $f$ a $C^{0,1}$ vector field on $\calX$. Then, for every $x_0 \in \calX$ there exists a unique Carath\'eodory solution $x: [0, T) \rightarrow \calX$ for some $T> 0$ that solves $\dot x(t) \in \tproj{\calX}{g}{f}(x(t)))$ with $x(0) = x_0$.
\end{theorem}

In conclusion, thanks to our coordinate-free definition of projected dynamical systems, our existence and uniqueness results seamlessly extend to systems defined on abstract manifolds.

\section{Stability of Projected Gradient Flows}\label{sec:stab}

To illustrate how established stability concepts easily apply to Krasovskii solutions of projected dynamical systems, we consider projected gradient systems, i.e., projected dynamical systems for which the vector field is the gradient of a function. Naturally, these systems are of prime interest for constrained optimization. Similar techniques can also be used to assess the stability of equilibria of other vector fields ranging from saddle-point flows~\cite{cherukuriAsymptoticConvergenceConstrained2016} to momentum methods~\cite{wilsonLyapunovAnalysisMomentum2016}. In what follows, we will establish convergence and stability results that generalize the work in~\cite{hauswirthProjectedGradientDescent2016}.

For simplicity, we consider systems defined on $\bbR^n$. Extensions to subsets of manifolds are possible using the results from \cref{sec:mfd} (see \cref{rem:grad_mfd} below).

\subsection{Preliminaries and LaSalle Invariance}
We quickly review some basic terminology for continuous-time systems defined by a constrained differential inclusion
\begin{align}\label{eq:gen_inclusion}
	\dot x \in F(x) \quad x \in \calX \, ,
\end{align}
where $\calX \subset \bbR^n$ is closed and $F: \calX \rightrightarrows \bbR^n$ is non-empty, closed, convex, locally bounded, and outer semicontinuous. In the following, a \emph{solution of \eqref{eq:gen_inclusion}} refers to a Carath\'eodory solution of \eqref{eq:gen_inclusion}, whereas a \emph{Krasovskii solution} of \eqref{eq:gen_inclusion} is a (Carath\'eodory) solution of the inclusion obtained from regularizing \eqref{eq:gen_inclusion}.

The \emph{$\omega$-limit set} of a complete solution $x$ of \eqref{eq:gen_inclusion} is the set of all points $\hat{x}$ for which there exists a sequence $\{t_k \}$ with $\lim_{k \rightarrow \infty} t_k = \infty$ and $\lim_{k \rightarrow \infty} x(t_k) = \hat{x}$.
A set $\calA \subset \calX$ is \emph{weakly invariant}, if for every initial condition $x_0  \in \calA$, there exists a complete solution starting at $x_0$ that remains in $\calA$ for all $t \geq 0$. The union of any weakly invariant subsets is weakly invariant, hence the notion of \emph{largest weakly invariant set} is well-defined.
A set $\calA \subset \calX$ is \emph{invariant}, if for every initial condition $x_0  \in \calA$, every complete solution starting at $x_0$ remains in $\calA$ for all $t \geq 0$.

Also recall that $\hat{x} \in \calX$ is a \emph{weak equilibrium} for \eqref{eq:gen_inclusion} if and only if $x(t) = \hat{x}$ for all $t \geq 0$ is a solution. Namely, $\hat{x}$ is a weak equilibrium if and only if $0 \in F(\hat{x})$.
A \emph{strong equilibrium} is a point $\hat{x}$ such that $x(t) = \hat{x}$ for all $t \geq 0$ is the only solution starting at $\hat{x}$.

A compact set $\calA \subset \calX$ is \emph{stable} for \eqref{eq:gen_inclusion} if for every (relative) neighborhood $\calV$ of $\calA$ there exists a neighborhood $\calW$ of $\calA$ such that every complete solution of \eqref{eq:gen_inclusion} starting in $\calW$ satisfies $x(t) \in \calV $ for all $t \geq 0$. The set $\calA$ is \emph{locally asymptotically stable}, if it is stable and there exists $\delta > 0$ such that every solution $x$ with $d_\calA(x(0)) \leq \delta$ converges to $\calA$, i.e., $\lim_{t \rightarrow \infty} d_\calA(x(t)) = 0$.

We will make use of the following invariance principle for differential inclusions. The result is a special case of \cite[Thm.~8.2]{goebelHybridDynamicalSystems2012} which applies to hybrid systems. For similar results for differential inclusions see also \cite{aubinViabilityTheory1991,ryanIntegralInvariancePrinciple1998}.

\begin{theorem}\label{thm:invar}
	Consider a continuous function $V: \bbR^n \rightarrow \bbR$, any function $u: \bbR^n \rightarrow [- \infty, \infty]$, and a set $\calU \subset \bbR^n$ such that $u(x) \leq 0$ for every $x \in \calU$ and such that the growth of $V$ along solutions of \eqref{eq:gen_inclusion} is bounded by $u$ on $\calU$. In other words, any solution $x: [0, T) \rightarrow \calU$ of~\eqref{eq:gen_inclusion} satisfies $V(x(t_1)) - V(x(t_0)) \leq \int_{t_0}^{t_1} u ( x(\tau)) d\tau$ for any $t_0, t_1 \in [0, T)$ and $t_0 < t_1$.
	Let $x$ be a complete and bounded solution of~\eqref{eq:gen_inclusion} such that $x(t) \in \calU$ for all $t \geq 0$. Then, for some $r \in V(\calU)$, $x$ approaches the nonempty set that is the largest weakly invariant subset of $V^{-1}(r) \cap \calU \cap \cl u^{-1}(0)$.
\end{theorem}

\subsection{Convergence of Projected Gradient Flows}
In the following, we consider projected gradient flows of the form
\begin{align}\label{eq:proj_grad_intro}
	\dot x \in \tproj{\calX}{g}{- \grad_g \Phi}(x)
\end{align}
where $\calX \subset \bbR^n$ is closed, and $g$ is a locally weakly bounded metric on $\calX$. Further, $\Phi: \bbR^n \rightarrow \bbR$ is an objective function, continuously differentiable in a neighborhood of $\calX$.
The \emph{gradient} $\grad_g \Phi(x)$ of $\Phi$ with respect to $g$ at $x \in \calX$ is the unique vector that satisfies $\left \langle \grad_g \Phi(x), w \right \rangle_{g(x)} = D_x \Phi (w)$ for all $w \in T_x \bbR^n$. In matrix notation we have
\begin{align*}
	\grad_g \Phi(x) = G^{-1}(x) \nabla \Phi(x)^T \, .
\end{align*}

The results of the previous sections can be used to guarantee the existence and uniqueness of (Carath\'eodory or Krasovskii) solutions of \eqref{eq:proj_grad_intro} under appropriate condtions on $\calX, g$ and $\Phi$. In fact, \eqref{eq:proj_grad_intro} is well-defined on subsets of abstract $C^{1}$-manifolds.

Dynamics of the form \eqref{eq:proj_grad_intro} serve to find local solutions of the constrained problem
\begin{equation}\label{eq:min_prob_intro}
	\text{minimize } \, \Phi(x) \quad \text{subject to } \, x \in \calX \,.
\end{equation}
A \emph{(strict) local minimizer of \eqref{eq:min_prob_intro}} is a point $x^\star \in \calX$ such that there exists a relative neighborhood $\calN \subset \calX$ of $x^\star$ and $\Phi(y) \geq (>) \Phi(x^\star)$ holds for all $y \in \calN \setminus \{ x^\star\}$.
A \emph{critical point} of \eqref{eq:min_prob_intro} is a point $x^\star \in \calX$ satisfying
\begin{align}\label{eq:equil_crit_pt}
	D_{x^\star} \Phi ( w ) = \nabla \Phi(x^\star) w \geq 0 \qquad \text{for all} \qquad w \in T_{x^\star} \calX \,.
\end{align}
Every local minimizer of \eqref{eq:min_prob_intro} is a critical point \cite[Thm.~6.12]{rockafellarVariationalAnalysis2009}. Further, if $\calX$ is Clarke regular and of the same form as in \cref{ex:clarke_reg_constraint_set}, \eqref{eq:equil_crit_pt} is equivalent to the well-known Karush-Kuhn-Tucker (KKT) conditions~\cite[Chap. 4]{bazaraaNonlinearProgrammingTheory2006}.

The metric $g$ is a property of the system~\eqref{eq:proj_grad_intro} only and does not affect the optimizers of~\eqref{eq:min_prob_intro}.
Furthermore, it is reasonable (but important to note) that, in general, the metric that defines the gradient has to be the same metric that defines the projection.
A particular choice of $g$ is, for example, induced by the Hessian of $\Phi$ if $\Phi$ is twice continuously differentiable and strongly convex. This leads to Newton-type dynamics (\cref{ex:newton} below).

When considering the projected gradient flow \eqref{eq:proj_grad_intro} we need to distinguish between equilibrium points for Carath\'eodory and Krasovskii solutions. In particular, we say that $x^\star$ is a \emph{weak (strong) K-equilibrium}, if it is a weak (strong) equilibrium of the Krasovskii-regularized inclusion. Analogously, $x^\star$ is a \emph{weak (strong) C-equilibrium} if it is an equilibrium for Carath\'eodory solutions (i.e., solutions of the unregularized inclusion).

Since every Carath\'eodory solution of \eqref{eq:proj_grad_intro} is also a Krasovskii solutions, it follows that every strong K-equilibrium is also a strong C-equilibrium. On the other hand, a weak C-equilibrium is a weak K-equilibrium.

We can now establish the relation between critical points and minimizers of \eqref{eq:min_prob_intro}, and the different types of equilibria of \eqref{eq:proj_grad_intro}.

\begin{lemma}\label{lem:weak_k_equil}
	Every critical point of \eqref{eq:min_prob_intro} is a weak K-equilibrium of \eqref{eq:proj_grad_intro}, and every weak C-equilibrium of \eqref{eq:proj_grad_intro} is a critical point of \eqref{eq:min_prob_intro}.
\end{lemma}

\begin{proof}
	Let $x^\star$ be a critical point of \eqref{eq:min_prob_intro}. By definition of $\grad_g \Phi$, we can reformulate \eqref{eq:equil_crit_pt} as $\left\langle - \grad_g \Phi(x^\star),  w \right\rangle_{g(x^\star)} \leq 0$ for all $w \in T_{x^\star} \calX$. Furthermore, by \cref{lem:kras_normal}, we have, for all $x \in \calX$,
	\begin{align*}
		\left\langle - \grad_g \Phi(x), w \right\rangle_{g(x)} \geq \| w \|^2_{g(x)} \quad \forall w\in  \Kras{\tproj{\calX}{g}{-\grad_g \Phi}}(x).
	\end{align*}
	Combining these two statements we get
	\begin{align*}
		0 \geq \left\langle - \grad_g \Phi(x^\star), w \right\rangle_{g(x^\star)} \geq \| w \|^2_{g(x^\star)}  \quad \forall w \in T_{x^\star} \calX \cap \Kras{\tproj{\calX}{g}{-\grad_g \Phi}}(x^\star).
	\end{align*}
	We know that $T_{x} \calX \cap \Kras{\tproj{\calX}{g}{-\grad_g \Phi}}(x) \neq \emptyset$ holds for all $x \in \calX$ by viability of $\tproj{\calX}{g}{-\grad_g \Phi}$. Therefore, we conclude that $w = 0 \in \Kras{\tproj{\calX}{g}{-\grad_g \Phi}}(x^\star)$ and $x^\star$ is a weak K-equilibrium.

	Next, assume that $x^\star \in \calX$ is a weak C-equilibrium, i.e., $0 \in \tproj{\calX}{g}{- \grad_g \Phi}(x^\star)$. If $x^\star$ were not a critical point of \eqref{eq:min_prob_intro}, then $\left\langle - \grad_g \Phi(x^\star), v \right\rangle_{g(x^\star)} > 0$ holds for some $v \in T_{x^\star} \calX$. This, however, means that  $0 \notin \tproj{\calX}{g}{- \grad_g \Phi}(x^\star)$.
	To see this, note that the projection of $u := - \grad_g \Phi(x^\star)$ onto the ray/cone spanned by $v$ is given by $w := ({\left\langle u, v\right\rangle_{g(x^\star)}}/{\| v\|^2_{g(x^\star)}}) v$ (note that $\left\langle u, v\right\rangle_{g(x^\star)} \geq 0$). Applying the Pythagorean theorem to the right triangle $\{0, u, w\}$, we have $\| u - w \| < \| u - 0 \|$. Hence, $0$ cannot be a projection of $u$ onto $T_{x^\star} \calX$ since it does not achieve the minimal distance to $T_{x^\star} \calX$ which contradicts the fact that $x^\star$ is a C-equilibrium.
\end{proof}

\begin{lemma}\label{lem:invar_k_grad}
	Along Krasovskii solutions of \eqref{eq:proj_grad_intro}, $\Phi$ is nonincreasing and, consequently, the sublevel sets $S_\ell := \{ x\, | \, \Phi(x) \leq \ell \} \cap \calX$ for $\ell \in \bbR$ are invariant.
\end{lemma}

\begin{proof}
	Given any Krasovskii solution $x: [0, T) \rightarrow \calX$ of \eqref{eq:proj_grad_intro}, for almost all $t \in [0, T)$ there exists  $w(t) \in \Kras{\tproj{\calX}{g}{- \grad_g \Phi } }(x(t))$ such that
	\begin{align*}
		\tfrac{d}{dt} \Phi(x(t)) = D_{x(t)} \Phi(w(t)) =  \left\langle \grad_g \Phi(x),  w(t) \right\rangle_{g(x(t))} \, .
	\end{align*}
	Using \cref{lem:kras_normal} on \cpageref{lem:kras_normal}, we then have
	\begin{align}\label{eq:grad_liederiv}
		\tfrac{d}{dt} \Phi(x(t)) =  - \left\langle -\grad_g \Phi(x(t)), w(t) \right\rangle_{g(x(t))}
		\leq - \| w(t) \|^2_{g(x(t))} \leq 0 \, .
	\end{align}
	Thus $\Phi$ is non-increasing along Krasovskii solutions of \eqref{eq:proj_grad_intro} and hence $\calS_\ell$ is invariant.
\end{proof}

\begin{lemma}\label{lem:strong_k_equil}
	Every local minimizer of \eqref{eq:min_prob_intro} is a strong K-equilibrium of \eqref{eq:proj_grad_intro}.
\end{lemma}

\begin{proof}
	By \cref{thm:main_exist}, there exists a Krasovskii solution $x: [0, T) \rightarrow \calX$ of \eqref{eq:proj_grad_intro} starting at the local minimizer $x^\star \in \calX$ of \eqref{eq:min_prob_intro}. Assume for the sake of contradiction that $x(0) = x^\star$ but $x(T) \neq x^\star$.
	The sublevel set $\calS_{\ell^\star}$ with $\ell^\star := \Phi(x^\star)$ is invariant and $x(t) \in \calS_{\ell^\star}$ for all $t \in [0, T)$, by \cref{lem:invar_k_grad}. Since $x^\star \in \calX$ is a local minimizer there exists a neighborhood $\calN \subset \calX$ of $x^\star$ such that $\Phi(x') \geq \Phi(x^\star)$ for all $x' \in \calN$. If necessary, restrict the solution $x$ such that $x : [0, T) \rightarrow \calN$. We have $\Phi(x(t)) = \Phi(x^\star)$ and $\tfrac{d}{dt} \Phi(x(t)) = 0$ for all $t \in [0, T)$, and therefore, for almost all $t \in [0, T)$, we have
	\begin{align*}
		0 = \tfrac{d}{dt} \Phi(x(t)) = - \left\langle -\grad_g \Phi(x(t)), \dot{x}(t) \right\rangle_{g(x)} \leq - \| \dot{x}(t) \|^2 \, ,
	\end{align*}
	where the inequality follows from \cref{lem:kras_normal}. Consequently, we have $\dot{x}(t) = 0$ for almost all $t \in [0, T)$ and thus $x(T) = \int_0^T \dot{x}(t) dt = x^\star$, establishing the contradiction.
\end{proof}

\Cref{lem:weak_k_equil,lem:strong_k_equil,thm:main_equiv} can be summarized as follows:
\begin{proposition}[connection between equilibria]\label{prop:proj_grad_cd}
	Consider the projected gradient flow \eqref{eq:proj_grad_intro} and the problem \eqref{eq:min_prob_intro}. The following inclusions hold:
	\begin{multline*}
		\text{local minimizer} \, \, \subset \, \,
		\text{strong K-eq.} \, \, \subset \, \,
		\text{strong C-eq.} \, \, \\ \subset \, \,
		\text{weak C-eq.} \, \, \subset \, \,
		\text{critical pt.} \, \, \subset \, \,
		\text{weak K-eq.}
	\end{multline*}
	If, in addition, $\calX$ is Clarke regular and $g$ is continuous, then we have
	\begin{align*}
		\text{local minimizer} \, \, \subset \, \,
		\text{strong eq.} \, \, \subset \, \,
		\text{weak eq.} \, \,  = \, \,
		\text{critical pt.}
	\end{align*}
\end{proposition}

If solutions of \eqref{eq:proj_grad_intro} are unique we do not distinguish between weak and strong equilibria and \cref{prop:proj_grad_cd} simplifies to equivalence of critical points and equilibria.

As an example of a critical point that is a weak (C-)equilibrium, but not a strong (C-)equilibrium we refer back to \cref{ex:prox_uniq} which illustrates this case for $\Phi(x) := x_1$. In that example, non-unique solutions may leave the critical point at arbitrary times, but the constant function is nevertheless a solution.

Unfortunately, convergence is generally guaranteed only to the set of weak K-equilibria as the following application of the invariance principle \cref{thm:invar} shows.

\begin{proposition}\label{prop:pgs_stab}
	Consider \eqref{eq:proj_grad_intro} and let $\Phi: \bbR^n \rightarrow \bbR$ have compact sublevel sets on $\calX$, i.e., for every $\ell \in \bbR$ the set $\mathcal{S}_\ell := \{ x\, | \, \Phi(x) \leq \ell \} \cap \calX$ is compact. Then, \eqref{eq:proj_grad_intro} admits a complete Krasovskii solution $x: [0, \infty) \rightarrow \calX$ for every initial condition $x(0) \in \calX$. Furthermore, for some $r \in \Phi(\calS_{\ell})$, $x$ converges to set of weak K-equilibrium points in $\Phi^{-1}(r) \cap \calX$.
	If, in addition, $\calX$ is Clarke regular and $g$ is continuous, then convergence is to the set of critical points of \eqref{eq:min_prob_intro}.
\end{proposition}

\begin{proof}
	We consider the Krasovskii regularization of \eqref{eq:proj_grad_intro} which is non-empty, closed, convex, locally bounded, and outer semicontinuous. As before, the compactness and invariance of the sublevel sets $\calS_\ell$ of $\Phi$ on $\calX$ implies that (Krasovskii) solutions cannot escape to the horizon in finite time and therefore must be complete.
	Hence, \cref{thm:invar} guarantees convergence to the largest weakly invariant subset for which $\tfrac{d}{dt} \Phi(x(t)) = 0$ (and which lies on a level set of $\Phi$ relative to $\calX$). Using \cref{eq:grad_liederiv}, we know that every limit point $\hat{x}$ of $x$ satisfies $0 \in \Kras{\tproj{\calX}{g}{- \grad_g \Phi}}(\hat{x})$, i.e., $\hat{x}$ is a weak K-equilibrium of~\eqref{eq:proj_grad_intro}. Finally, under Clarke regularity of $\calX$ and continuity of $g$, \cref{prop:proj_grad_cd} implies that every weak equilibrium is a critical point.
\end{proof}

Although convergence is generally only to weak equilibria, the following theorem, inspired by \cite{absilStableEquilibriumPoints2006}, establishes the connection between stability and optimality.

\begin{theorem}[stability \& optimality]\label{thm:stab}
	Consider \eqref{eq:proj_grad_intro} and let $\Phi$ have compact sublevel sets on $\calX$ as in \cref{prop:pgs_stab}. For some $r$, let $\hat{\calX} \subset \{x \in \calX \, | \, \Phi(x) = r \}$ be a connected set of weak K-equilibria. Then, the following statements hold:
	\begin{enumerate}[label = (\roman*)]
		\item\label{enum:stab1} If $\hat{\calX}$ is locally asymptotically stable for~\eqref{eq:proj_grad_intro} then it is a \emph{strict set of minimizer of~\eqref{eq:min_prob_intro}}, i.e., $\Phi(y) > r$ for all $y \in \calN \setminus \hat{\calX}$ where $\calN$ is a neighborhood of $\hat{\calX}$.
		\item\label{enum:stab2} If $\hat{\calX}$ is a strict set of minimizers of \eqref{eq:min_prob_intro} then it is stable for~\eqref{eq:proj_grad_intro}.
	\end{enumerate}
\end{theorem}

\begin{proof}
	Recall from \cref{prop:pgs_stab} that the compactness of the sublevel sets of $\Phi$ guarantees the existence of complete solutions.
	To show~\ref{enum:stab1}, let $\calV \subset \calX$ be a neighborhood of $\hat{\calX}$ such that any solution $x: [0, \infty) \rightarrow \calX$ of~\eqref{eq:proj_grad_intro} with $x(0) \in \calV$ converges to $\hat{\calX}$. Since $\Phi$ is $C^1$ and $x$ is absolutely continuous, $\Phi \circ x$ is absolutely continuous, and we may write
	\begin{equation*}
		\underset{t \rightarrow +\infty}{\lim}(\Phi \circ x)(t)  = \Phi( x(0)) + \int\nolimits_{0}^{+\infty} D_x \Phi(\dot x(t)) dt =  r \, .
	\end{equation*}
	Since $D_x \Phi(\dot{x}(t)) \leq 0$ holds for almost all $t \geq 0$, it follows that $\int_{0}^{+\infty}  D_x \Phi(\dot x(t)) dt \leq 0$, and hence $r \leq \Phi(x(0))$ for all $t \geq 0$. Because this reasoning applies to all $x(0)$ in the region of attraction of $\hat{\calX}$, it follows that $\hat{\calX}$ is a local minimizer of $\Phi$.

	To see that $\hat{\calX}$ is a strict local minimizer, assume for the sake of contradiction that for some $\widetilde{x}$ in the region of attraction of $\hat{\calX}$ it holds that $\Phi(\widetilde{x}) \leq r$.
	Every solution $y$ to~\eqref{eq:proj_grad_intro} with $y(0) = \widetilde{x}$ nevertheless converges to $\hat{\calX}$ by assumption. Therefore, it must hold that $\int_{0}^{+\infty} D_y \Phi(\dot{y}(t)) = 0$ and since $D_y \Phi(\dot{y}(t)) \leq 0$, it follows that $D_y \Phi(\dot{y}(t)) = 0$ for almost all $t \geq 0$.
	But as a consequence of \cref{prop:pgs_stab}, all points in the $\omega$-limit set are weak K-equilibrium points, this holds in particular for $\widetilde{x}$ and therefore $\hat{\calX}$ cannot be locally asymptotically stable.

	For~\ref{enum:stab2}, assume that $\hat{\calX} \neq \calX$ (otherwise stability is trivial). Hence, consider a bounded (relative) neighborhood $\calW \subset \calX$ of $\hat{\calX}$ in which $\hat{\calX}$ is a strict local minimizer.
	Next, we construct a neighborhood $\calV \subset \calW$ such that all trajectories starting in $\calV$ remain in $\calW$. Namely, let $\alpha$ be such that $r < \alpha < {\min}_{x \in \partial \calW} \Phi(x)$ where $\partial \calW$ is the boundary of~$\calW$ relative to $\calX$. Define $\calV := \{ x \in \calW\, | \, \Phi(x) \leq \alpha \} \subseteq \calW$ which has a non-empty interior because $r < \alpha$. Since for any trajectory, we have $D_x \Phi(\dot x(\tau)) \leq 0$ we conclude that $\calV$ is strongly invariant and remains in $\calV$, thus establishing stability.
\end{proof}

It is not possible to draw stronger conclusions (e.g., that strict minimizers are always locally asymptotically stable) than in \cref{thm:stab}, unless additional assumptions are satisfied. A counter-example for an unconstrained gradient flow (which is, technically, a special case of a projected gradient flow) is documented in \cite{absilStableEquilibriumPoints2006}.

\begin{remark}\label{rem:grad_mfd}
	The results of this section can be generalized to projected gradient flows on $C^1$ manifolds.
	For instance, since any (equilibrium or critical) point under consideration can be locally mapped into $\bbR^n$, \cref{prop:proj_grad_cd} applies directly to projected gradient flows on manifolds. Similarly, the statements of \cref{thm:stab} about the relation between stability and optimality hold true on manifolds, especially if the set $\hat{\calX}$ of weak K-equilibria is contained in a single chart domain.
	On the other hand, because \cref{prop:pgs_stab} is a global statement, for it to generalize to manifolds an invariance principle akin to \cref{thm:invar} but for differential inclusion on manifolds is required. Such a generalization is plausible, but has not yet been documented.
\end{remark}

\begin{remark}
	Projected gradient flows like \eqref{eq:proj_grad_intro} can be approximated (or implemented) in different ways. On one hand, standard numerical integration schemes can be adapted for (Euclidean) projected dynamical systems on convex domains as documented in \cite{nagurneyProjectedDynamicalSystems1996},  yielding well-known numerical optimization algorithms. In the non-Euclidean, non-convex setting, oblique projected gradient flows can be implemented, e.g., as in \cite{haberleNonconvexFeedbackOptimization2020} by linearizing constraints around the current state. This leads to algorithms similar to sequential quadratic programming schemes~\cite{nocedalNumericalOptimization2006}.
	Another possibility are \emph{anti-windup approximations} \cite{hauswirthAntiWindupApproximationsOblique2020a,hauswirthDifferentiabilityProjectedTrajectories2020a} which serve to implement projected dynamical systems as the closed-loop behavior of feedback control loops that are subject to input saturation in \emph{feedback-based optimization} \cite{bernsteinOnlinePrimalDualMethods2019,colombinoOnlineOptimizationFeedback2019,hauswirthProjectedGradientDescent2016}.
\end{remark}

As a specific example of a projected gradient flow, we consider the metric $g$ to be the Hessian of the objective function, resulting in a \emph{projected Newton flow}:
\begin{example}\label{ex:newton}
	Let $\calX \subset \bbR^n$ be closed, and let $\Psi: \bbR^n \rightarrow \bbR$ be strongly convex and globally Lipschitz continuous and twice differentiable. In particular, the Hessian of $\Psi$ (denoted by $\nabla^2 \Psi$) is continuous and has lower and upper bounded eigenvalues. Therefore, we may use $\nabla^2 \Psi$ to define the weakly bounded metric $\left\langle u , v \right\rangle_{g(x)} := u^T \nabla^2 \Psi(x) v$ for $u, v \in T_x \bbR^n$. Hence, the projected gradient flow
	\begin{align}\label{eq:newton_flow}
		\dot{x} \in \tproj{\calX}{g}{\left ( - \grad_g \Psi \right)}(x) \, , \quad x(0) = x_0 \in \calX
	\end{align}
	where $\grad_g \Psi(x) = {(\nabla^2 \Psi(x))}^{-1} \nabla \Psi(x)^T$ is a constrained form of a \emph{Newton flow}, i.e., the continuous-time limit of the well-known \emph{Newton method} for optimization.
	If $\calX$ is convex, one can recover a a \emph{proximal Newton-type method} \cite{leeProximalNewtonTypeMethods2014} for solving \eqref{eq:min_prob_intro} as a projected forward Euler discretization of \eqref{eq:newton_flow} (possibly with variable step size).
\end{example}

\subsection{Connection to Subgradient Flows}

Assuming that $f$ is the gradient field of an objective function and $\calX$ is Clarke regular, we can establish the connection between oblique projected gradients and subgradients. This fact is well-known for convex functions (and lesser known for regular functions~\cite{clarkeNonsmoothAnalysisControl1998,cortesDiscontinuousDynamicalSystems2008}) in the Euclidean metric, but, as we show next, generalizes to a variable metric.

Recall that $\Psi: \calV \rightarrow \overline{\bbR}$, where $\calV \subset \bbR^n$ is open and $\overline{\bbR} := \bbR \cup \{ \infty \}$, is \emph{(subdifferentially) regular} if its epigraph $\epi \Psi := \{ (x, y) \, | \, x\in \calV, \, y \geq \Psi(x) \}$ is non-empty and Clarke regular.

\begin{definition}
	Given a metric $g$ on an open set $\calV \subset \bbR^n$ and a regular function $\Psi: \calV \rightarrow \overline{\bbR}$, $v$ is a \emph{subgradient of $\Psi$ with respect to $g$ at $x$}, denoted by $v \in \partial\Psi(x)$, if
	\begin{align*}
		\underset{y \rightarrow x}{\lim \inf} \, \tfrac{\Psi(y) - \Psi(x) - \left\langle v, y - x \right\rangle_{g(x)}}{\| y - x \|} \geq 0 \, .
	\end{align*}
\end{definition}
Namely, if $\Psi$ is differentiable at $x$, then $\partial \Psi(x) = \{ \grad_g \Psi(x) \}$. Further, if $\calX \subset \calV$ is Clarke regular and $I_\calX: \calV \rightarrow \overline{\bbR}$ denotes its indicator function, then $\partial I_\calX(x) = N^g_x \calX$.

The next result is a direct combination of~\cite[Ex.~8.14]{rockafellarVariationalAnalysis2009} and~\cite[Cor.~10.9]{rockafellarVariationalAnalysis2009}.

\begin{proposition}
	Let $\hat \Psi := \Psi + I_\calX$ where $\Psi: \calV \rightarrow \bbR$ is a $C^1$ function and $I_\calX$ is the indicator function of a Clarke regular set $\calX \subset \calV$ where $\calV \subset \bbR^n$ is open. Then, for all $x \in \calX$ one has
	\begin{align*}
		\partial \hat \Psi(x) = \grad_g \Psi(x) + N^g_x \calX \, .
	\end{align*}
\end{proposition}

It follows immediately from \cref{cor:equiv_normal} that under the appropriate assumptions trajectories of projected gradient flows are also solutions to subgradient flows.

\begin{corollary}[equivalence with subgradient flows]\label{cor:subgrad_equiv}
	Let $\calX$ be Clarke regular, let $g$ be a continuous metric on $\calX$, and let $\Psi$ be a $C^1$ objective function on an open neighborhood of $\calX$. Then, for any $x_0 \in \calX$ there exists a Carath\'eodory solution $x: [0, T) \rightarrow \calX$ to the subgradient flow
	\begin{align*}
		\dot x \in - \partial( \Psi + I_\calX)(x)\,, \quad x(0) \in \calX \, .
	\end{align*}
	Furthermore, $x$ is a solution if and only if it is a Carath\'eodory (and Krasovskii) solution to the projected gradient flow \eqref{eq:proj_grad_intro}.
\end{corollary}

In summary, we have seen that projected gradient flows are well-defined in very general settings if one considers Krasovskii solutions. The convergence behavior is more fine-grained than for special cases (e.g., convex optimization problems) since the notion of equilibrium depends on the definition of the solution concept. Further, projected gradient flows exhibit the same connection between stability and optimality of equilibria as unconstrained gradient flows. Finally, \emph{oblique} projected gradient flows on Clarke regular sets can be interpreted subgradient flows of a composite function that is the sum of a smooth objective and the indicator function of the feasible set.

\section{Conclusion}\label{sec:conclusion}

\begin{table}[bt]
	\makegapedcells
	\centering
	{\footnotesize
		\begin{tabular}{lrrrrl}
								& $f$           & $g$            & $\calX$   & $\calM$                      & \\
			\toprule
			\makecell[cl]{Local Existence of Krasovskii                                             \\ solutions} &
			LB        &
			LWB       & loc.\ compact & $C^1$          &
			\makecell[cl]{Thm.~\ref{thm:main_exist}                                                 \\ Thm.~\ref{thm:main_mfd}} \\ \midrule
			\makecell[cl]{Global Existence of Krasovskii                                            \\ solutions (multiple possibilities)} &
			$C^0$     & $C^0$         & compact        & $C^1$     &
			Cor.~\ref{cor:max_sol}                                                                  \\ \midrule
			\makecell[cl]{Equivalence of Krasovskii                                                 \\ and Carath\'eodory solutions} &
			$C^0$     & $C^0$         & Clarke regular & $C^1$     &
			\makecell[cl]{Thm.~\ref{thm:main_equiv}                                                 \\ Thm.~\ref{thm:main_mfd}} \\ \midrule
			\makecell[cl]{Equivalence of projected gradient                                         \\ and subgradient flows} &
			$C^0$     & $C^0$         & Clarke regular & $C^1$     & Cor.~\ref{cor:subgrad_equiv}   \\ \midrule
			\makecell[cl]{Uniqueness of (Krasovskii \&                                              \\ Carath\'eodory) solutions} &
			$C^{0,1}$ & $C^{0,1}$     & prox-regular   & $C^{1,1}$ &
			\makecell[cl]{Thm.~\ref{thm:main_uniq}                                                  \\ Thm.~\ref{thm:main_mfd_uniq}}\\ \bottomrule
		\end{tabular}
	}
	\caption{Summary of results: regularity requirements for projected dynamical systems for a vector field $f$, metric $g$, feasible domain $\calX$ and regularity of the manifold $\calM$. (LB\@: locally bounded; LWB\@: locally weakly bounded)}\label{tab:summary}
	\vspace{-.2cm}
\end{table}

We have provided an extensive study of projected dynamical systems on irregular subset on manifolds, including the model of oblique projection directions. We have carved out sharp regularity requirements on the feasible domain, vector field, metric and differentiable structure that are required for the existence, uniqueness and other properties of solution trajectories. \cref{tab:summary} summarizes these results. In the process, we have established auxiliary findings, such as the fact that prox-regularity is an intrinsic property of subset of $C^{1,1}$ manifolds and independent of the choice of Riemannian metric.

While we believe these results are of general interest in the context of discontinuous dynamical systems, they particularly provide a solid foundation for the study of continuous-time constrained optimization algorithms for nonlinear, nonconvex problems. To illustrate this point, we have included a study the stability and convergence of Krasovskii solutions to projected gradient descent---arguably the most prototypical continuous-time constrained optimization algorithm.

\section*{Acknowledgments}
We would like to thank Gabriela Hug and Matthias Rungger for their support in putting together this paper.

\appendix

\section{Technical definitions and results}\label{app:basic_notions}

\unless\ifARXIV
	The following technical lemmas are required for different results in the current paper. More background material can be found in the appendix of~\cite{hauswirthProjectedDynamicalSystems2018a}.
\fi

\begin{lemma}\label{lem:tgt_deriv} Given a set $\calX \subset \bbR^n$, for any absolutely continuous function $x: [0, T) \rightarrow \calX$ with $T>0$ it holds that $\dot x(t) \in T_{x(t)} \calX \cap -T_{x(t)} \calX$ almost everywhere on $[0, T)$, where $-T_{x(t)} := \{ v | -v \in T_{x(t)} \}$.
\end{lemma}

\begin{proof} Let $t \in [0, T)$ be such that $\dot x(t)$ exists. This implies that by definition
	\begin{equation*}
		\dot x(t) = \underset{\tau \rightarrow 0^+}{\lim} \tfrac{x(t+ \tau) - x(t)}{\tau} = \underset{\tau \rightarrow 0^+}{\lim} \tfrac{x(t) - x(t-\tau) }{\tau},
		\quad
	\end{equation*}
	Thus, by choosing any sequence $\tau_k \rightarrow 0$ with $\tau_k > 0$, the sequence $\frac{x(t+ \tau_k) - x(t)}{\tau_k}$ converges to a tangent vector and $\frac{ - x(t-\tau_k) + x(t)}{\tau_k}$ converges to a vector in $-T_{x(t)} \calX$ by definition of $T_{x(t)} \calX$ and the fact that $x(t) \in \calX$ for all $t \in [0, T)$.
\end{proof}

The following is a local version of~\cite[Lem.~1.30]{peypouquetConvexOptimizationNormed2015}.
\begin{lemma}[Descent Lemma]\label{lem:c11_lipschitz}
	Let $\Phi: V \rightarrow \bbR$ be a $C^{1,1}$ map where $V \subset \bbR^n$ is open. Given $x \in V$ there exists $L > 0$ such that for all $z,y \in V$ in a neighborhood of $x$ it holds that
	\begin{align*}
		| \Phi(z) - \Phi(y) - D_y \Phi(z - y) | \leq L \| z - y \|^2
	\end{align*}
\end{lemma}

\ifARXIV
	For a comprehensive treatment of the following definitions and results see~\cite{rockafellarVariationalAnalysis2009,aubinDifferentialInclusionsSetValued1984,peypouquetConvexOptimizationNormed2015,hiriart-urrutyFundamentalsConvexAnalysis2012}.
	Given a sequence $\{x_{k}\}$ and a set $\mathcal X$, the notation $x_k \overset{sub}{\underset{\calX} \longrightarrow} x$ denotes the existence of a subsequence $\{x_{k'}\}$ that converges to $x$ and $x_{k'} \in \calX$ for all $k'$. Similarly, $x_k \overset{ev}{\underset{\calX} \longrightarrow} x$ implies that $x_k \in \calX$ holds \emph{eventually}, i.e., for all $k$ larger than some $K$, and that $\{x_k\}$ converges to $x$.
	Given a sequence of sets $\{C_k\}$ in $\bbR^n$, its \emph{outer limit} and \emph{inner limit} are given as
	\begin{align*}
		\underset{k \rightarrow \infty}{\lim \sup} \, C_k
		 & := \left\lbrace x \, \middle|\,
		\exists \{x_i\}: x_i \underset{C_i}{\overset{sub}{\longrightarrow}} x
		\right\rbrace \quad \mbox{and} \quad
		\underset{k \rightarrow \infty}{\lim \inf} \, C_k
		 & := \left\lbrace x \, \middle|\,
		\exists \{x_i\}: x_i \underset{C_i}{\overset{ev}{\longrightarrow}} x
		\right\rbrace \
	\end{align*}
	respectively. As a pedagogical example to distinguish between inner and outer limits, consider an alternating sequence of sets given by
	$C_{2m} := A$ and $C_{2m+1} := B$. Then, we have ${\lim \sup}_{k \rightarrow \infty} \, C_k = A \cup B$ and $ {\lim \inf}_{k \rightarrow \infty} \, C_k = A \cap B$. On the one hand any constant sequence $\{x_k\}$ with $x_k = c \in A \cap B$ for all $k$ satisfies the requirement such that $c \in {\lim \inf}_{k \rightarrow \infty} \, C_k$. On the other hand, any sequence $\{x_k\}$ with $x_{2m} = a \in A$ for $m \in \mathbb{N}$ has a trivial (constant) subsequence converging to $a \in A$ and hence $a \in {\lim \sup}_{k \rightarrow \infty} \, C_k $. The following result relates the image of an outer (inner) limit to the outer (inner) limit of images of a map $f$.
\fi

\begin{lemma}{\cite[Thm.~4.26]{rockafellarVariationalAnalysis2009}}\label{lem:continuity_set}
	For a sequence of sets $\{C_k \}$ in $V \subset \bbR^n$ and a continuous map $f: V \rightarrow \bbR^m$, one has
	\begin{equation*}
		f\left(\underset{k \rightarrow \infty}{\lim \inf} \, C_k \right)
		\subset \underset{k \rightarrow \infty}{\lim \inf} \, f(C_k) \, , \qquad   f\left(\underset{k \rightarrow \infty}{\lim \sup} \, C_k \right)
		\subset \underset{k \rightarrow \infty}{\lim \sup} \, f(C_k) \, .
	\end{equation*}
\end{lemma}

For a set-valued map $F: V \rightrightarrows W$ with $V \subset \bbR^n$ and $W \subset \bbR^m$ its \emph{outer limit} and \emph{inner limit} at $x$ are defined respectively as
\begin{equation*}
	\underset{y \rightarrow x}{\lim \sup} \, F(y) := \bigcup_{x_k \underset{V}{\longrightarrow} x} \underset{k \rightarrow \infty}{\lim \sup} \, F(x_k) \quad \mbox{and} \quad
	\underset{y \rightarrow x}{\lim \inf} \, F(y) := \bigcap_{x_k \underset{V}{\longrightarrow} x} \underset{k \rightarrow \infty}{\lim \inf} \, F(x_k) \, .
\end{equation*}

\ifARXIV
	A set-valued map $F : V \rightrightarrows \bbR^m$ for $V \subset \bbR^n$ is \emph{outer semicontinuous (osc) at $x \in V$} if ${\lim \sup}_{y \rightarrow x} \, F(y) \subset F(x)$~\cite[Def.~5.4]{rockafellarVariationalAnalysis2009}.
	The map $F$ is \emph{upper semicontinuous (usc) at $x$} if for any open neighborhood $A \subset V$ of $ F(x)$ there exists a neighborhood $B \subset V$ of $x$ such that for all $ y\in B$ one has $F(y) \subset A$~\cite[Def.~2.1.2]{aubinViabilityTheory1991}. The map $F$ is \emph{outer (upper) semi-continuous} if and only if it is osc (usc) at every $x \in V$.
	For locally bounded, closed set-valued maps outer and upper semicontinuity are equivalent.

	\begin{lemma}{\cite[Lem.~5.15]{goebelHybridDynamicalSystems2012}}\label{lem:outer_sem_closedgraph}
		Let $F : V \rightrightarrows \bbR^m$ be closed and locally bounded for $V \subset \bbR^n$. Then, $F$ is osc at $x \in V$ if and only if it is usc at $x$.
		Furthermore, $F$ is osc/usc at $x$ if and only if
		$\gph F := \left\lbrace (x,v) \, \middle|\, x\in V, v \in F(x) \right\rbrace$
		locally closed at $x$.
	\end{lemma}

	The next result states that upper semicontinuity is preserved by convexification.

	\begin{lemma}{\cite[Lem.~16, \S 5]{filippovDifferentialEquationsDiscontinuous1988}}\label{lem:filippov_convex}
		Given a set-valued map $F: V \rightrightarrows \bbR^m$ with $V \subset \bbR^n$, if $F$ is usc and $F(x)$ is non-empty and compact for each $x \in V$, then the map $\co F: V \rightrightarrows \bbR^m$ defined as $x \mapsto \co F(x)$ is usc.
	\end{lemma}

	The following result is a generalization of~\cite[Prop.~6.5]{rockafellarVariationalAnalysis2009} to the case of a continuous metric instead of the standard Euclidean metric:

	\begin{lemma}\label{lem:normal_outer_semi}
		Let $\calX$ be Clarke regular. If the metric $g$ on $\calX$ is continuous, then the set-valued map $\calX \mapsto N^g_x \calX$ is outer semi-continuous.
	\end{lemma}

	\begin{proof} Consider any two sequences $x_k \rightarrow x$ with $x_k \in \calX$ and $\eta_k \rightarrow \eta$ with $\eta_k \in N^g_{x_k} \calX$. To complete the proof we need to show that $\eta \in N^g_x \calX$.
		By definition of $N^g_{x_k} \calX$ we have $\left\langle v, \eta_k \right\rangle_{g(x_k)} \leq 0$ for all $v \in T^C_x \calX$. Furthermore, by continuity of $g$ we have $\left\langle v, \eta \right\rangle_{g(x)} \leq 0$ for all $v \in {\lim \sup}_{x_k \rightarrow x} \, T_{x_k}^C \calX$. (Namely, we must have $\left\langle v_k, \eta_k \right\rangle_{g(x_k)} \leq 0$ for every sequence $v_k \rightarrow v$ with $v_k \in T^C_{x_k} \calX$, hence the use of $\lim \sup$.) By definition of the Clarke tangent cone, we note that $\left\langle v, \eta \right\rangle_{g(x)} \leq 0$ holds for all
		\begin{equation*}
			v \in T_x^C \calX = \underset{x_k \rightarrow x}{\lim \inf} \, T_{x_k} \calX =
			\underset{x_k \rightarrow x}{\lim \inf} \, T_{x_k}^C \calX \subset
			\underset{x_k \rightarrow x}{\lim \sup} \, T_{x_k}^C \calX \, ,
		\end{equation*}
		and therefore $\eta \in N^g_x \calX$.
	\end{proof}

	The following general existence and viability theorem goes back to~\cite{haddadMonotoneTrajectoriesDifferential1981}. Similar results can also be found in~\cite{aubinViabilityTheory1991,clarkeOptimizationNonsmoothAnalysis1990,goebelHybridDynamicalSystems2012}.

	\begin{proposition}[{\cite[Cor.~1.1, Rem 3]{haddadMonotoneTrajectoriesDifferential1981}}]\label{prop:haddad}
		Let $\calX$ be a locally compact subset of $\bbR^n$ and $F: \calX \rightrightarrows \bbR^n$ an usc, non-empty, convex and compact set-valued map. Then, for any $x_0 \in \calX$ there exists $T > 0$ and a Lipschitz continuous function $x: [0, T) \rightarrow \calX$ such that $x(0) = x_0$ and $\dot x(t) \in F(x(t))$ almost everywhere in [0, T) if and only if the condition $F(x) \cap T_x \calX \neq \emptyset$ holds for all $x \in \calX$.
		Furthermore, for $r>0$ such that $U_r := \{ x \in \calX \, | \, \| x - x_0 \| \leq r \}$ is closed and $L = \max_{y \in U_r} \| F(y) \|$ exists, the solution is Lipschitz and exists for $T > r/ L$.
	\end{proposition}
\fi

\bibliographystyle{siamplain}
\bibliography{bibliography_new}
\end{document}